\documentclass[11pt]{amsart}
\usepackage{tikz,amsthm,amsmath,amstext,amssymb,amscd,epsfig,euscript, mathrsfs, dsfont,pspicture,multicol,graphpap,graphics,graphicx,times,enumerate,subfig,sidecap,wrapfig,color}
\usepackage{amsmath}
\usepackage{amssymb}
\usepackage{pgf}

\usepackage[colorlinks,citecolor=red]{hyperref}

\newcommand{\R}{{\mathbb R}}       
\newcommand{\Z}{{\mathbb Z}}       
\newcommand{\DD}{{\mathcal D}}

\newcommand{\BB}{{\mathcal B}}
\newcommand{\FF}{{\mathcal F}}
\newcommand{\HH}{{\mathcal H}}

\newcommand{\LL}{{\mathcal L}}

\newcommand{\BZ}{{\mathcal B}}

\newcommand{\RR}{{\mathcal R}}

\newcommand{\EE}{{\mathcal E}}

\newcommand{\TT}{{\mathcal T}}

\newcommand{\hm}{{\omega}}

\newcommand{\diam}{\mathop{\rm diam}}
\newcommand{\dist}{{\rm dist}}

\newcommand{\rf}[1]{{(\ref{#1})}}

\newcommand{\supp}{\operatorname{supp}}

\newcommand{\vphi}{{\varphi}}
\newcommand{\ve}{{\varepsilon}}
\newcommand{\vv}{{\vspace{2mm}}}
\newcommand{\vvv}{\vspace{4mm}}
\newcommand{\wt}[1]{{\widetilde{#1}}}
\newcommand{\wh}[1]{{\widehat{#1}}}

\newcommand{\HD}{{\mathsf{HD}}}
\newcommand{\LD}{{\mathsf{LD}}}
\newcommand{\sss}{{\mathsf{Stop}}}
\newcommand{\ttt}{{\mathsf{Top}}}
\newcommand{\tree}{{\mathsf{Tree}}}
\newcommand{\reg}{{\mathsf{Reg}}}
\newcommand{\nex}{{\mathsf{Next}}}
\newcommand{\ch}{{\mathsf{ch}}}

\def\Xint#1{\mathchoice
{\XXint\displaystyle\textstyle{#1}}%
{\XXint\textstyle\scriptstyle{#1}}%
{\XXint\scriptstyle\scriptscriptstyle{#1}}%
{\XXint\scriptscriptstyle\scriptscriptstyle{#1}}%
\!\int}
\def\XXint#1#2#3{{\setbox0=\hbox{$#1{#2#3}{\int}$ }
\vcenter{\hbox{$#2#3$ }}\kern-.58\wd0}}

\def\avint{\Xint-}

\newtheorem{theorem}{Theorem}[section]
\newtheorem{lemma}[theorem]{Lemma}

\newtheorem{coro}[theorem]{Corollary}

\newtheorem{propo}[theorem]{Proposition}

\newtheorem*{lemma*}{Lemma}
\newtheorem*{theorem*}{Theorem}

\theoremstyle{definition}

\theoremstyle{remark}
\newtheorem{rem}[theorem]{\bf Remark}

\numberwithin{equation}{section}

\newcommand{\RRem}{\begin{rem}}
\newcommand{\erem}{\end{rem}}

\def\d{\partial}

\makeatletter
\def\@tocline#1#2#3#4#5#6#7{\relax
  \ifnum #1>\c@tocdepth 
  \else
    \par \addpenalty\@secpenalty\addvspace{#2}%
    \begingroup \hyphenpenalty\@M
    \@ifempty{#4}{%
      \@tempdima\csname r@tocindent\number#1\endcsname\relax
    }{%
      \@tempdima#4\relax
    }%
    \parindent\z@ \leftskip#3\relax \advance\leftskip\@tempdima\relax
    \rightskip\@pnumwidth plus4em \parfillskip-\@pnumwidth
    #5\leavevmode\hskip-\@tempdima
      \ifcase #1
       \or\or \hskip 1em \or \hskip 2em \else \hskip 3em \fi%
      #6\nobreak\relax
    \dotfill\hbox to\@pnumwidth{\@tocpagenum{#7}}\par
    \nobreak
    \endgroup
  \fi}
\makeatother

\def\Cap{\mathop\mathrm{Cap}}

\def\cM{{\mathcal{M}}}

\begin{document}

\title[Uniform rectifiability and bounded harmonic functions]{Uniform rectifiability from Carleson measure estimates and $\boldsymbol\ve$-approximability of bounded harmonic functions}



\newcommand{\jonas}[1]{\marginpar{\color{magenta} \scriptsize \textbf{Jonas:} #1}}

\author[Garnett]{John Garnett}

\address{John Garnett\\
Department of Mathematics, 520 Portola Plaza \\University of California,
 Los Angeles, Los Angeles, California 90095-1555.
}
\email{jbg@mat.ucla.edu}

\author[Mourgoglou]{Mihalis Mourgoglou}

\address{Mihalis Mourgoglou
\\
Departament de Matem\`atiques\\
 Universitat Aut\`onoma de Barcelona
\\
Edifici C Facultat de Ci\`encies
\\
08193 Bellaterra (Barcelona), Catalonia
}

\curraddr{\sc{BCAM - Basque Center for Applied Mathematics}\\
\sc{Mazarredo, 14 E48009 Bilbao, Basque Country--Spain.}
}
\email{mourgoglou@mat.uab.cat}

\author[Tolsa]{Xavier Tolsa}
\address{Xavier Tolsa
\\
ICREA, Passeig Lluís Companys 23 08010 Barcelona, Catalonia, and\\
Departament de Matem\`atiques and BGSMath
\\
Universitat Aut\`onoma de Barcelona
\\
Edifici C Facultat de Ci\`encies
\\
08193 Bellaterra (Barcelona), Catalonia
}
\email{xtolsa@mat.uab.cat}

\thanks{J.G. was supported by NSF Grant DMS 1217239. M.M was supported  by the ERC grant 320501 of the European Research Council (FP7/2007-2013) and also by the Basque Government through the BERC 2014-2017 program and by Spanish Ministry of Economy and Competitiveness MINECO: BCAM Severo Ochoa excellence accreditation SEV-2013-0323. X.T. was  supported by the ERC grant 320501 of t7/2007-2013) and also by 2014-SGR-75 (Catalonia), MTM2013-44304-P (Spain), and the Marie Curie ITN MAnET (FP7-607647).
}

\begin{abstract}
Let $\Omega\subset\R^{n+1}$, $n\geq1$, be a corkscrew domain  with Ahlfors-David regular boundary. In this paper we prove that $\partial\Omega$ is
uniformly $n$-rectifiable if every bounded harmonic function on $\Omega$ is 
$\varepsilon$-approximable or if every bounded harmonic function on $\Omega$ 
 satisfies a suitable square-function Carleson measure estimate.  
In particular, this applies to the case when $\Omega=\R^{n+1}\setminus E$ and $E$ is  Ahlfors-David regular.
Our results solve a conjecture posed by Hofmann, Martell, and Mayboroda in a recent work where they 
proved the converse statements.
Here we also obtain two additional criteria for uniform rectifiability. One is given in terms of the so-called ``$S<N$'' estimates, and another in terms of a suitable corona decomposition involving harmonic measure. 
\end{abstract}

\maketitle



\section{Introduction}

In this paper we  characterize the uniform $n$-rectifiability of the boundary of a domain $\Omega \subset \R^{n+1}$, $n\geq1$, in terms of a square function
Carleson measure estimate, or of an approximation property, for the bounded harmonic functions on $\Omega.$
Our results solve an open problem posed by Hofmann, Martell and Mayboroda in \cite{HMM2}.

We introduce some definitions and notations.
A set $E\subset \R^d$ is called $n$-{\textit {rectifiable}} if there are Lipschitz maps
$f_i:\R^n\to\R^d$, $i=1,2,\ldots$, such that 
\begin{equation}\label{eq001}
\HH^n\biggl(E\setminus\bigcup_i f_i(\R^n)\biggr) = 0,
\end{equation}
where $\HH^n$ stands for the $n$-dimensional Hausdorff measure. 
A set $E\subset\R^{d}$ is called $n$-AD-{\textit {regular}} (or just AD-regular or Ahlfors-David regular) if there exists some
constant $c>0$ such that
$$c_0^{-1}r^n\leq \HH^n(B(x,r)\cap E)\leq c_0\,r^n\quad \mbox{ for all $x\in
E$ and $0<r\leq \diam(E)$.}$$
 
The set $E\subset\R^{d}$ is  {\textit {uniformly}}  $n$-{\textit {rectifiable}} if it is 
$n$-AD-regular and
there exist constants $\theta, M >0$ such that for all $x \in E$ and all $0<r\leq \diam(E)$ 
there is a Lipschitz mapping $g$ from the ball $B_n(0,r)$ in $\R^{n}$ to $\R^d$ with $\text{Lip}(g) \leq M$ such that
$$
\HH^n (E\cap B(x,r)\cap g(B_{n}(0,r)))\geq \theta r^{n}.$$

The analogous notions for  measures are the following. 
A Radon measure $\mu$ on $\R^d$ is $n$-{\textit {rectifiable}} if it vanishes outside  an $n$-rectifiable
set $E\subset\R^d$ and if moreover $\mu$ is absolutely continuous with respect to $\HH^n|_E$.
On the other hand, $\mu$ is called $n$-AD-regular if it is of the form $\mu=g\,\HH^n|_E$, where 
$E$ is $n$-AD-regular and $g:E\to (0,+\infty)$ satisfies $g(x)\approx1 $ for all $x\in E$, with the implicit constant independent of $x$.
If, moreover, $E$ is uniformly $n$-{\textit {rectifiable}}, then $\mu$ is called 
{\textit{uniformly}} $n$-{\textit{rectifiable}}.

The notion of uniform rectifiability should be considered a quantitative version of rectifiability. It 
was introduced in their pioneering works \cite{DS1}, \cite{DS2} of David and Semmes, who were seeking 
a good geometric framework under which all singular integrals with odd and sufficiently smooth kernel are bounded in $L^2$. 

An open set $\Omega\subset\R^{n+1}$ is called a {\textit {corkscrew domain}} if for every ball $B(x,r)$ with
 $x\in\partial\Omega$ and $0<r\leq\diam(\Omega)$ there exists another ball $B(x',r')\subset \Omega\cap B(x,r)$
 with radius $r'\approx r$, with the implicit constant independent of $x$ and $r$. 
Let us remark that we do not ask $\Omega$ to be connected. For example, if $E\subset\R^{n+1}$
is a closed $n$-AD-regular set, then it follows easily that $\R^{n+1}\setminus E$ is a corkscrew domain.

Let  $\Omega \subset \R^{n+1}$ be open, and let
$u$ be  a 
bounded harmonic function on $\Omega$.  For $\ve > 0$ we say that
$u$ is 
  $\ve$-\textit{approximable} 
if there is $\varphi \in W^{1,1}_{\rm {loc}}(\Omega)$
and $C>0$ such that 
\begin{equation}
\label{epapp}
\|u - \varphi\|_{L^{\infty}(\Omega)} < \varepsilon
\end{equation}
\and for all $x \in \partial\Omega$ and all $r > 0$
\begin{equation}
\label{epCar}
\frac{1}{r^n} \int_{B(x,r)} |\nabla \varphi(y)| \,dy \leq C,
\end{equation}
where $dy$ denotes the Lebesgue measure in $\R^{n+1}.$
It is clear by a normal family argument that every bounded harmonic function on $\Omega$ is 
$\varepsilon$-approximable for all $\varepsilon >0$ if and only if \eqref{epapp} and
 \eqref{epCar} hold for all harmonic $u$ with 
$||u||_{L^{\infty}(\Omega)} \leq 1$ with constant $C = C_{\varepsilon}$ depending on $\varepsilon$ 
but not on $u$. The notion of $\ve$-approximability was introduced by Varopoulos in \cite{Var} in connection
with 
 corona problems.   See \cite[Chapter VIII]{Garnett} for a proof on the upper half plane and \cite{Dah},
 for the case of Lipschitz domains.  Also see \cite{Garnett}, \cite{HMM2}, \cite{KKiPT}, \cite {KKoPT}, and \cite{Pipher} for further results 
 and  applications, including some to elliptic operators.  

Our main result is the following:
\begin{theorem} \label{teo1}
Let $\Omega\subset\R^{n+1}$, $n\geq1$, be a corkscrew domain with $n$-AD-regular boundary.  Then the following are equivalent:
\begin{itemize}
\item[(a)] $\partial\Omega$ is uniformly $n$-rectifiable.
\vv
\item[(b)] Every bounded harmonic function on $\Omega$ is $\varepsilon$-approximable for all $\varepsilon  > 0.$ 
\vv
\item[(c)] There is $C>0$ such that if $u$ is a bounded harmonic function on $\Omega$ and $B$ is a ball centered at $\partial\Omega$,
\begin{equation}\label{eqhip1}
\int_B |\nabla u(x)|^2\,\dist(x,\partial\Omega)\,dx\leq C\,\|u\|^2_{L^\infty(\Omega)}\,r(B)^n.
\end{equation}
\end{itemize}

\end{theorem}


The implications (a) $\Rightarrow$ (b) and (a) $\Rightarrow$ (c) have already been  proved by Hofmann, Martell, and
Mayboroda in \cite{HMM2} for $n \geq 2$, but a careful reading of their proof shows same implications hold for $n =1$ 
with small modifications.  
In the current paper we will only prove that (b) $\Rightarrow$ (a) and that (c) $\Rightarrow$ (a),
but in a  
slightly stronger formulation  because   we only  assume (b) or (c) holds for bounded functions continuous on $\overline \Omega$ and harmonic
on $\Omega.$

As a corollary of the preceding theorem we deduce another characterization of uniform rectifiability in terms
of a square function - nontangential maximal function estimate (of the type ``$S<N$") in the case $n\geq2$. To state this result we need some additional notation. Given $x\in\partial\Omega$, we define the {\textit {cone}}
$$\Gamma(x) = \{y\in\Omega:|x-y| < 2\,\dist(y,\partial\Omega)\}$$
and for a continuous function $u$ in $\Omega$, we define the {\textit {non-tangential maximal function}}
$$N_* u(x) = \sup_{y\in\Gamma(x)}|u(y)|.$$
For $u\in W^{1,2}_{loc}(\Omega)$ we also define the {\textit {square function}}
$$Su(x) = \left(\int_{y\in\Gamma(x)}|\nabla u(y)|^2\,\dist(y,\partial\Omega)^{1-n}\,dy\right)^{1/2}.$$

Then we have:

\begin{coro}\label{coro1}
Let $\Omega\subset\R^{n+1}$, $n\geq2$, be a corkscrew domain with $n$-AD-regular boundary. 
Denote by $\mu$ the surface measure on $\partial\Omega$.
Suppose that for some $p \in [2, \infty)$ there exists some constant $C_p>0$ such that for every function $u \in C_0(\overline \Omega)$ harmonic in $\Omega$,
\begin{equation}\label{eqhld2}
\|Su\|_{L^p(\mu)} 
\leq C_p\,\|N_*u\|_{L^p(\mu)}\quad 
\end{equation}
Then $\partial\Omega$ is uniformly rectifiable.
\end{coro}

In Hofmann and Le \cite[eq. (4.12)]{Hofmann-Le} the estimate
\rf{eqhld2} (at least for $n \geq 2$) is asserted for corkscrew domains with uniformly $n$-rectifiable boundaries and attributed to a forthcoming paper by Hofmann, Martell and Mayboroda.  From that paper and Corollary 1.2 it follows 
 that \rf{eqhld2} also characterizes uniform
rectifiability for corkscrew domain with $n$-AD-regular boundary for $n \geq 2.$ 

\vv

We will prove Theorem \ref{teo1} by using the connection between harmonic measure  the Riesz transforms and then applying the result from \cite{NToV} that the $L^2(\mu)$ boundedness of the vector or Riesz transforms implies the
uniform $n$-rectfiability of $\mu$. The connection between harmonic measure and uniform
$n$-rectifiability has been a subject of intensive research in the last years. See for example, 
\cite{DJ}, \cite{HM1}, \cite{HMU}, \cite{HMM1}, \cite{AHMNT}, \cite{BH}, and \cite{HM2}. Among these we would like to highlight \cite{HM1} and \cite{HMU}, from which it follows that for a bounded uniform domain $\Omega\subset\R^{n+1}$ (so that $\partial\Omega$ is $n$-AD-regular), the harmonic measure $\omega^p$ in $\Omega$ is an
$A_\infty$ weight with respect to the surface measure if and only if $\partial\Omega$ is uniformly $n$-rectifiable. On the other hand, the connection between harmonic measure and the Riesz transforms,
in combination with the rectifiability criteria from \cite{NToV} and \cite{NToV-pubmat},
has been successfully exploited
in other recent works such as \cite{AHM3TV}, \cite{MTo}, and \cite{AMT}.

To show that the Riesz transform vector is bounded in $L^2(\mu)$ we will use a corona type decomposition. Unlike the usual corona decompositions of David and Semmes in \cite{DS2}, which are of geometric nature, 
the one we will need is based on the comparison between surface measure and harmonic measure. We will derive the usual packing condition for such decomposition from the assumptions (b) or (c) in Theorem \ref{teo1} with a suitable test function $u$ whose construction involves the harmonic measure.
Then, by a comparison argument between  surface measure and harmonic measure we will prove the
boundedness of the Riesz transform ``at the scale of each tree'' of the corona decomposition. To implement this argument
is a non-trivial task as, for example, both measures may be mutually singular. 
To overcome these technical difficulties we  will use
the suppressed Riesz kernels introduced by Nazarov, Treil and Volberg in \cite{NTV} and the sophisticated $Tb$ theorems they proved for such operators.

\vv
Finally, we remark that the corona decomposition we use in the proof of Theorem \ref{teo1} provides a new
characterization, in terms of harmonic measure,  of  uniform rectifiability for boundaries of corkscrew domains.  This complements, in some sense, another  characterization in terms of big pieces of NTA domains,
recently obtained by Bortz and Hofmann  \cite{BH} and Martell and
 Hofmann \cite{HM2}.
This corona decomposition characterization  is described in  Propositions \ref{propo11} and \ref{propo12}, but an equivalent and somewhat less technical form of it is provided by the following theorem.

\begin{theorem}\label{teo2}
Let $\Omega\subset\R^{n+1}$ be a corkscrew domain with $n$-AD-regular boundary. Denote by
$\mu$ the surface measure on $\partial\Omega$, and let $\DD_\mu$ be a dyadic lattice of cubes associated to $\mu$ as in Subsection 
\ref{subsec21}.
Then $\partial\Omega$ is uniformly $n$-rectifiable if and only if there exists a family $\FF\subset\DD_\mu$ 
satisfying the following properties:
\begin{itemize}
\item[(a)] Every cube $Q\in\DD_\mu$ is contained in some cube $R\in\FF$.

\item[(b)] 
The family $\FF$ fulfills the packing condition
$$\sum_{R\subset S: R\in\FF}\mu(R)\leq C\,\mu(S)\quad \mbox{for all $S\in\DD_\mu$}.$$

\item[(c)] For each 
$R\in\FF$ there exists a corkscrew point $p_R\in\Omega$ with 
$$c^{-1}\ell(R)\leq\dist(p_R,R)\leq \dist(p_R,\partial\Omega)\leq c\,\ell(R)$$
such that, if $R$ is the smallest cube from $\FF$ containing some cube $Q\in \DD_\mu$, then
$$
\omega^{p_R}(5Q)\approx \frac{\mu(Q)}{\mu(R)},$$
with the implicit constant uniform on $Q$ and $R$.
\end{itemize}
\end{theorem}

\vvv

\section{Preliminaries}

As usual in harmonic analysis, we denote by $C$ or $c$ constants which usually only depend on the dimension $n$ and
other fixed parameters (such as the constants involved in the AD-regularity of $\partial\Omega$ or the corkscrew condition of $\Omega$), and which may change their values at different occurrences. On the contrary, constants with subscripts such as $c_0$ or $C_0$, do not change their values.
For $a,b\geq 0$, we will write $a\lesssim b$ if there is $C>0$ so that $a\leq Cb$ and $a\lesssim_{t} b$ if the constant $C$ depends on the parameter $t$. We write $a\approx b$ to mean $a\lesssim b\lesssim a$ and define $a\approx_{t}b$ similarly.

\vv
\subsection{Dyadic lattices}\label{subsec21}
Given an $n$-AD-regular measure $\mu$ in $\R^{n+1}$ we consider 
the dyadic lattice of ``cubes'' built by David and Semmes in \cite[Chapter 3 of Part I]{DS2}. These dyadic cubes are not true cubes, but they play the role of cubes with respect to a given $n$-AD-regular measure $\mu$. 
The properties satisfied by $\DD_\mu$ are the following. 
Assume first, for simplicity, that $\diam(\supp\mu)=\infty$. Then for each $j\in\Z$ there exists a family $\DD_{\mu,j}$ of Borel subsets of $\supp\mu$ (the dyadic cubes of the $j$-th generation) such that:
\begin{itemize}
\item[$(a)$] each $\DD_{\mu,j}$ is a partition of $\supp\mu$, i.e.\ $\supp\mu=\bigcup_{Q\in \DD_{\mu,j}} Q$ and $Q\cap Q'=\varnothing$ whenever $Q,Q'\in\DD_{\mu,j}$ and
$Q\neq Q'$;
\vv
\item[$(b)$] if $Q\in\DD_{\mu,j}$ and $Q'\in\DD_{\mu,k}$ with $k\leq j$, then either $Q\subset Q'$ or $Q\cap Q'=\varnothing$;
\vv
\item[$(c)$] for all $j\in\Z$ and $Q\in\DD_{\mu,j}$, we have $2^{-j}\lesssim\diam(Q)\leq2^{-j}$ and $\mu(Q)\approx 2^{-jn}$;
\vv
\item[$(d)$] there exists $C>0$ such that, for all $j\in\Z$, $Q\in\DD_{\mu,j}$, and $0<\tau<1$,
\begin{equation}\label{small boundary condition}
\begin{split}
\mu\big(\{x\in Q:\, &\dist(x,\supp\mu\setminus Q)\leq\tau2^{-j}\}\big)\\&+\mu\big(\{x\in \supp\mu\setminus Q:\, \dist(x,Q)\leq\tau2^{-j}\}\big)\leq C\tau^{1/C}2^{-jn}.
\end{split}
\end{equation}
\end{itemize}
Property (d) is often called the {\em small boundaries condition}.
From (\ref{small boundary condition}), it follows that there is a point $z_Q\in Q$ (the center of $Q$) such that $\dist(z_Q,\supp\mu\setminus Q)\gtrsim 2^{-j}$ (see \cite[Lemma 3.5 of Part I]{DS2}).
We set $\DD_\mu:=\bigcup_{j\in\Z}\DD_{\mu,j}$. 

In the case $\diam(\supp\mu)<\infty$, the families $\DD_{\mu,j}$ are only defined for $j\geq j_0$, with
$2^{-j_0}\approx \diam(\supp\mu)$, and the same properties above hold for $\DD_\mu:=\bigcup_{j\geq j_0}\DD_{\mu,j}$.

Given a cube $Q\in\DD_{\mu,j}$, we say that its side length is $2^{-j}$, and we denote it by $\ell(Q)$. Notice that $\diam(Q)\leq\ell(Q)$. 
We also denote 
\begin{equation}\label{defbq}
B_Q:=B(z_Q,c_1\ell(Q)),
\end{equation}
where $c_1>0$ is some fix constant so that $B_Q\cap\supp\mu\subset Q$, for all $Q\in\DD_\mu$.

For $\lambda>1$, we write
$$\lambda Q = \bigl\{x\in \supp\mu:\, \dist(x,Q)\leq (\lambda-1)\,\ell(Q)\bigr\}.$$

\vv
\subsection{The Riesz transform and harmonic measure}

Given a Radon measure $\mu$ in $\R^{n+1}$, its $n$-{\textit {dimensional Riesz transform}} is defined by
$$\RR\mu(x) = \int \frac{x-y}{|x-y|^{n+1}}\,d\mu(y),$$
whenever the integral makes sense.  For $\ve>0$, we also denote
$$\RR_\ve\mu(x) = \int_{|x-y|>\ve} \frac{x-y}{|x-y|^{n+1}}\,d\mu(y),\qquad
\RR_*\mu(x)= \sup_{\ve>0}\bigl|\RR_\ve \mu(x)\bigr|.$$
For $f\in L^1_{loc}(\mu)$, we write  $\RR_\mu f\equiv \RR(f\mu)$, 
$\RR_{\mu,\ve} f\equiv \RR_\ve(f\mu)$, and $\RR_{\mu,*} f\equiv \RR_*(f\mu)$. 
We say that $\RR_\mu$ is bounded in $L^2(\mu)$ if the
operators $\RR_{\mu,\ve}$ are bounded in $L^2(\mu)$ uniformly on $\ve>0$.
We will also use the {\textit {centered maximal Hardy-Littlewood operator}}
$$M_\mu f(x) =\sup_{r>0} \frac1{\mu(B(x,r))}\int |f|\,d\mu.$$

Let $\EE$ denote the fundamental solution for the Laplace equation in $\R^{n+1}$, so that $\mathcal{E}(x)=c_n\,|x|^{1-n}$ for $n\geq 2$, $c_n>0$.The {\it Green function} $G:\Omega\times \Omega\rightarrow[0,\infty]$ for an open set $\Omega\subset \R^{n+1}$ is a function with the following properties: for each $x\in \Omega$, 
$$G(x,y)=\EE(x-y)+h_{x}(y),$$ 
where $h_{x}$ is harmonic on $\Omega$, and whenever $v_{x}$ is a nonnegative superharmonic function that is the sum of $\EE(x-\cdot)$ and another superharmonic function, then  $v_{x}\geq G(x,\cdot)$, from which it follows that $G(x,y)$ is
unique (\cite[Definition 4.2.3]{H}). 

An open subset of $\R^{n+1}$ having  Green function is called a {\it Greenian} set. By \cite[Theorem 4.2.10]{H}, all open subsets of $\R^{n+1}$ are Greenian for $n\geq 2$. In the case $n=1$, if $\HH^1(\partial\Omega)>0$ for example, 
which is implied by the assumptions of Theorem \ref{teo1}, then $\Omega\subset\R^2$ is Greenian.

We denote by $\omega^p$ the harmonic measure in $\Omega$ with pole at $p\in\Omega$.
The Green function can be then written as follows (see \cite[Lemma 6.8.1]{AG}): for $x,y\in\Omega$, $x\neq y$
\begin{equation}\label{green}
G(x,y) = \mathcal{E}(x-y) - \int_{\partial\Omega} \mathcal{E}(x-z)\,d\omega^y(z).
\end{equation}
Notice that the Riesz transform has kernel 
\begin{equation}\label{eqker}
K(x) = c_n\,\nabla \EE(x),
\end{equation}
for a suitable absolute constant $c_n$, so that for $x, p \in\Omega$, by  \rf{green}  and \rf{eqker} we get
\begin{align}\label{eqclau1}
\RR\omega^{p}(x) = 
c_n\nabla_x\int \EE(x-y)\,d\omega^{p}(y) & = c_n\,\nabla_x\bigl(\EE(x-p) - G(x,{p})\bigr) \\
&= K(x-{p}) - c_n\,\nabla_x G(x,{p}).\nonumber
\end{align}
\vv

The following is a very standard result, usually known as Bourgain's estimate. See for example \cite{AHM3TV} for more details. 

\begin{lemma}
\label{lembourgain}
There is $\delta_{0}>0$ depending only on $n\geq 1$ so that the following holds for $\delta\in (0,\delta_{0}]$. If $\Omega\subsetneq \R^{n+1}$ is a domain, $\xi \in \partial \Omega$, $r>0$, and $B=B(\xi,r)$,
then for all $s>n-1$ and all $x \in \delta B$,  
\begin{equation}\label{Bour}
 \omega^{x}(B)\gtrsim_{n} \frac{\mathcal H_\infty^{s}(\partial\Omega\cap \delta B)}{(\delta r)^{s}}.
 \end{equation}
\end{lemma}

\begin{rem}
If $\mu$ is some measure supported on $\d \Omega$ such that $\mu(B(y,r))\leq C\,r^n$ for all $y$, $r>0$, then from the preceding lemma 
it follows that 
\begin{equation}\label{eq:BourFro}
 \omega^{x}(B)\gtrsim  \frac{\mu(\delta_0 B)}{(\delta_0 r)^{n}}\quad\mbox{  for all }x\in \delta B.
\end{equation}
\end{rem}

\vv

The next lemma is also standard. See for example \cite{AHM3TV} or \cite{AMT-porous} for the detailed proof.

\begin{lemma}\label{l:w>G}
Let $\Omega\subsetneq \R^{n+1}$, $\xi \in \partial\Omega$, $r>0$ and $B:=B(\xi,r)$.  Suppose that there exists a point $x_B \in \Omega$ so that the ball $B_0:=B(x_{B},r/C)$ satisfies $4B_0\subset \Omega\cap B$ for some $C>1$. Then in the case $n\geq 2$
the harmonic measure and Green function of $\Omega$ satisfy
\begin{equation}\label{eq:Green-lowerbound}
 \omega^{x}(B)\gtrsim \omega^{x_{B}}(B)\, r^{n-1}\, G(x,x_{B})\,\, \,\,\,\,\,\,\text{for all}\,\, x\in \Omega\backslash  B_0.
 \end{equation}
In the case $n=1$, if  $\Omega$ is Greenian then
\begin{equation}\label{eq:Green-lowerbound2}
 \omega^{x}(B)\gtrsim \omega^{x_{B}}(B)\, \bigl|G(x,x_{B})-
 G(x,z)\bigr|\quad \mbox{for all $x\in\Omega\setminus B_0$ and $z\in \frac12B_0$.}
 \end{equation}
 The implicit constants in \rf{eq:Green-lowerbound} and \rf{eq:Green-lowerbound2} depend only on $C$ and $n$.
\end{lemma}

\vv
If $\partial\Omega$ is $n$-AD-regular and $0<r(B)<\dfrac{\delta_0\diam(\Omega)}2$, then by Lemma \ref{l:w>G} and \eqref{eq:BourFro}, in the case $n\ge2$ we have
for all $x\in \Omega\backslash  2B$ and all $y \in B\cap\Omega$,
 \begin{equation}
 \label{e:w>g}
 \omega^{x}(2\delta_0^{-1}B)\gtrsim  r^{n-1}\,G(x,y)
 \end{equation} 
 Analogously, in the case $n =1$ we have
 for all $x\in \Omega\backslash  2B$ and $y,z\in B\cap\Omega$,
  \begin{equation}
 \label{e:w>gpla}
 \omega^{x}(2\delta_0^{-1}B)\gtrsim  r^{n-1}\, | G(x,y)-G(x,z)|
\end{equation}


\vv
\subsection{Uniform and NTA domains}

Following  \cite{JK}, we say that $\Omega$ satisfies the {\textit {Harnack chain condition}} if there is a constant $c$ such that
for every $\rho> 0$, $\Lambda\geq 1$, and every pair of points $x_1,x_2\in \Omega$ with $\dist(x_i,\partial\Omega)\geq\rho$
 for $i=1,2$ and $|x_1-x_2|<\Lambda\rho$, there is a chain of open balls $B_1,\ldots,B_N\subset\Omega$, with $N\leq C(\Lambda)$, with
$x_1\in B_1$, $x_2\in B_N$, $B_k\cap B_{k+1}\neq\varnothing$ and $\dist(B_k,\partial\Omega)\approx_c \diam(B_k)$
 for all $k$.
The preceding chain of balls is called a ``Harnack chain''.
A connected domain $\Omega\subset\R^{n+1}$ is called a {\textit {uniform domain}}
if it is a corkscrew domain and satisfies the Harnack chain condition.
Finally, $\Omega$ is called an {\textit {NTA domain}},  (for ``non-tangentially 
accessible") if
 if $\Omega$ is a uniform domain and the exterior $\R^{n+1} \setminus \overline \Omega$ is a non-empty corkscrew domain.

 Let $\Cap$ denote  logarithmic capacity if $n=1$ and  Newtonian capacity if $n\geq2$. A domain $\Omega\subset \R^{n+1}$ satisfies the {\it capacity density condition} (or CDC) if there are $R_{\Omega}>0$ and $c_\Omega>0$ such that  for any ball $B$ centered on $\d\Omega$ of radius $r(B)\in (0,R_{\Omega})$,
$$\Cap(B\backslash \Omega)\geq \left\{\begin{array}{ll}
c_\Omega \,r(B) &\quad \mbox{if $n =1$,}\\
\\
c_\Omega \,r(B)^{n-1} &\quad \mbox{if $n\geq 2$.}
\end{array}
\right.
$$
If $\partial \Omega$ is AD regular, then $\Omega$ satisfies the CDC. 


\vvv
\section{The corona decompositon for harmonic measure}\label{sec:corona}

In this section we will show that if either of the assumptions (b) or (c) in Theorem \ref{teo1} holds, then there exists a family $\FF$
having the properties described in Theorem \ref{teo2}. Later, in Section \ref{sec6}, we will show that the existence of such a family $\FF$
implies the uniform rectifiability of $\mu$. The proof of these facts will yield both Theorem \ref{teo1} and \ref{teo2}, because  Hofmann, Martell and Mayboroda have already shown in \cite{HMM2} that both (b) and (c) in Theorem \ref{teo1} hold if
$\mu$ is uniformly rectifiable.

We assume throughout this  section that $\Omega\subset\R^{n+1}$ is a corkscrew domain with $n$-AD-regular boundary, and that either  assumption (b) or (c) of Theorem \ref{teo1} holds.
We denote $\mu=\HH^n|_{\partial\Omega}$, and we consider the associated David-Semmes lattice $\DD_\mu$. 

\subsection{The corona decomposition}\label{subsecorona}

It will be convenient  to rephrase the properties of the required family $\FF$ in terms of a corona type decomposition.
A {\textit {corona decomposition}} of $\mu$ is a partition of $\DD_\mu$ into trees.
A family $\TT\subset\DD_\mu$ is a tree if it
verifies the following properties:
\begin{enumerate}
\item $\TT$ has a maximal element (with respect to inclusion) $Q(\TT)$ which contains all the other
elements of $\TT$ as subsets of $\R^{n+1}$. The cube $Q(\TT)$ is the ``root'' of $\TT$.
\vv
\item If $Q,Q'$ belong to $\TT$ and $Q\subset Q'$, then any $\mu$-cube $P\in\DD^\mu$ such that $Q\subset P\subset
Q'$ also belongs to $\TT$.
\vv
\item If $Q\in\TT$, then either all or none of the children of $Q$ belong to $\TT$.
\end{enumerate}
If $R=Q(\TT)$, we also write $\TT=\tree(R)$.

\vv
The precise result that we intend to prove in this section is the following.

\begin{propo}\label{propo11}
Let $\Omega\subset\R^{n+1}$ be a corkscrew domain with $n$-AD-regular boundary. Denote by
$\mu$ the surface measure on $\partial\Omega$.
Suppose that either the assumption (b) or (c) from Theorem \ref{teo1} holds.
Then $\mu$ admits a corona decomposition
$\DD_\mu=\bigcup_{R\in \ttt} \tree(R)$
so that the family $\ttt$ is a Carleson family, that is,
\begin{equation}\label{eqpack57}
\sum_{R\subset S: R\in\ttt}\mu(R)\leq C\,\mu(S)\quad \mbox{for all $S\in\DD_\mu$},
\end{equation}
and for each $R\in\ttt$ there exists a corkscrew point $p_R\in\Omega$ with 
$$c^{-1}\ell(R)\leq\dist(p_R,R)\leq \dist(p_R,\partial\Omega)\leq c\,\ell(R)$$
so that
\begin{equation}\label{eqpa87}
\omega^{p_R}(3Q)\approx \frac{\mu(Q)}{\mu(R)}\quad\mbox{ for all $Q\in\tree(R)$,}
\end{equation}
with the implicit constant uniform on $Q$ and $R$.
\end{propo}

It is easy to check that the existence of a corona decomposition such as the one in the proposition implies the existence of a family $\FF\subset\DD_\mu$ like the one described in Theorem \ref{teo2}.
Indeed, if the above corona decomposition exists we just take $\FF=\ttt$, and we can check that this satisfies the properties stated 
in Theorem \ref{teo2}, since \rf{eqpa87} also holds with $5Q$ replaced by $3Q$ (with a different implicit constant).
So Proposition \ref{propo11} proves one of the implications in Theorem \ref{teo2}.

\vv


\vv
\subsection{The approximation lemma}\label{sec:aproximation}

For any $Q\in\DD_\mu$, we consider a corkscrew point $p_Q\in B_Q\cap\Omega$. 
Recall that
$\omega^{p_Q}(Q)\gtrsim1$,
assuming that $p_Q$ has been chosen close enough to the center of $Q$, for example.
A more quantitative result is the following:

\begin{lemma}\label{lempq0}
There are constants $0<\alpha<1$ and $c_2>0$, depending only on $n$ and the AD-regularity constant of $\mu$ such that the
following holds. For any $0<\ve<1/2$ and any $Q\in\DD_\mu$, we have
$$\omega^x(Q)\geq \omega^x(\tfrac34B_Q)\geq1-c_2\ve^\alpha\quad\mbox{ if $x\in \frac12B_Q$\; and\; $\dist(x,\partial\Omega)\leq \ve\,\ell(Q)$.}$$
\end{lemma}

\begin{proof}
Since $\omega^x((\tfrac34B_Q)^c)$ is harmonic on $\Omega$, bounded by $1$, and vanishes on  $\partial\Omega\cap\frac34B_Q$, and since $\partial\Omega$  is $n$-AD-regular (and thus $\Omega$ satisfies the CDC),  there exists some $\alpha>0$ such that
$$\omega^x((\tfrac34B_Q)^c)\leq C\left(\frac{\dist(x,\partial\Omega)}{\ell(Q)}\right)^\alpha$$
if $x\in\frac12 B_Q$ (see for example Lemma 4.5 and Corollary 4.6 from \cite{AMT}). 
Therefore 
$\omega^x((\tfrac34B_Q)^c)\lesssim \ve^\alpha$ if $\dist(x,\partial\Omega)\leq \ve\,\ell(Q)$.
\end{proof}
\vv

From now on, we will assume that $p_Q\in \frac12B_Q\cap\Omega$, with 
\begin{equation}\label{eq1}
\dist(p_Q,\partial\Omega)\approx\ve\,\ell(Q),\quad \;\ve\ll1,
\end{equation}
so that
$\omega^{p_Q}(Q)\geq \omega^x(\tfrac34B_Q)\geq 1-C\,\ve^\alpha$. The corkscrew condition for $\Omega$ ensures the existence of such point $p_Q$.
We denote by $y_Q$ a point in $\partial\Omega$ such that
\begin{equation}\label{eq-1}
\dist(p_Q,\partial\Omega) = |y_Q-p_Q|,
\end{equation}
and we assume that $p_Q$ has been chosen so that 
\begin{equation}\label{eq0}
B(y_Q,|y_Q-p_Q|)\subset \tfrac34B_Q.
\end{equation}
We also denote 
$V_Q= B(p_Q,\frac1{10}\dist(p_Q,\partial\Omega))$, so that $V_Q\subset\Omega$. Notice that
$$r(V_Q)\approx \ve \,\ell(Q).$$
\vv

The next lemma is the main technical result of this subsection.

\begin{lemma}\label{lemapprox}
Suppose that the constant $\ve$ in \rf{eq1} is small enough.
Let $Q\in\DD_\mu$ and let $E_Q\subset Q$ be such that
$$\omega^{p_Q}(E_Q) \geq (1-\ve)\,\omega^{p_Q}(Q).$$
Then there exists a non-negative harmonic function $u_Q$ on $\Omega$ and a Borel function $f_Q$ with 
$$u_Q(x) = \int_{E_Q} f_Q\,d\omega^x,$$
$$f_Q\leq c\,\chi_{E_Q}\quad\!\! \mbox{if $n\geq2$, \quad and }\quad
f_Q\leq c\,|\log\ve|\,\chi_{E_Q}\quad\!\! \mbox{if $n=1$,}$$
and a unit vector $e_Q\in\R^{n+1}$ such that 
\begin{equation}\label{eqkey29}
\nabla u_Q(x)\cdot e_Q\geq c\,\frac1{r(V_Q)}\quad\mbox{ for all $x\in V_Q$.}
\end{equation}
In particular,
\begin{equation}\label{eqkey30}
\int_{V_Q} |\nabla u_Q(x)|^2\,\dist(x,\partial\Omega)\,dx\gtrsim r(V_Q)^n\approx_\ve \ell(Q)^n.
\end{equation}
\end{lemma}

\begin{proof}[Proof in the case $n\geq2$]
Let $y_Q\in\partial\Omega$ be the point defined in \rf{eq-1}. By rotating the domain if necessary we may assume that $p_Q-y_Q$ is parallel to the $x$ axis and that $p_{Q,1}>y_{Q,1}$.
Then, for all $x\in V_Q$ and all $y\in B(y_Q,r(V_Q))$,
$$0<\frac{x_1-y_1}{|x-y|^{n+1}} \approx \frac1{r(V_Q)^n}.$$
Therefore, if we take
\begin{equation}\label{deggg*}
g_Q(x) := \int_{B(y_Q,r(V_Q))} \frac1{r(V_Q)\,|x-y|^{n-1}}\,d\mu(y) \quad \mbox{ if $n\geq2$,}
\end{equation}
then we have
\begin{align*}
|\nabla g_Q(x)|& \geq -\partial_1 g_Q(x) = c\int_{B(y_Q,r(V_Q))} \frac{x_1-y_1}{r(V_Q)\,|x-y|^{n+1}}\,d\mu(y)\\
&\approx \frac{\mu(B(y_Q,r(V_Q))}{r(V_Q)^{n+1}} \approx \frac1{r(V_Q)}
\quad\mbox{ for all $x\in V_Q$.}
\end{align*}

By the AD-regularity of $\mu$, it is also immediate that
$\|g_Q\|_\infty\lesssim 1$. 
Then we define $f_Q:=\chi_{E_Q}\,g_Q$ and
$$u_Q(x) := \int f_Q\, d\omega^x = \int_{E_Q} g_Q\,d\omega^x.$$
To prove the estimate \rf{eqkey29} with $e_Q=-e_1$, first note that $g_Q$ is harmonic in $\Omega$ and continuous in $\R^{n+1}$, because of the local $\mu$ uniform integrability of $1/|x-y|^{n-1}$.
Thus, for all $x\in\Omega$,
$$g_Q(x) = \int g_Q\,d\omega^x,$$
and then,
\begin{equation}\label{eqggg*0}
\bigl|g_Q(x) - u_Q(x)\bigr| = \left|\int_{\partial\Omega\setminus E_Q} g_Q\,d\omega^x\right|
\leq \|g_Q\|_\infty\,\omega^x(\partial\Omega\setminus E_Q)\lesssim 
\,\omega^x(\partial\Omega\setminus E_Q).
\end{equation}
By \rf{eq1} and  the assumption in the lemma,
$$
\omega^{p_Q}(\partial\Omega\setminus E_Q) = 
\omega^{p_Q}(\partial\Omega\setminus Q) + \omega^{p_Q}(Q\setminus E_Q)
\leq C\ve^\alpha + \ve\lesssim \ve^\alpha,
$$
and then 
 by Harnack's inequality 
it follows that 
\begin{equation}\label{eqggg*1}
\omega^{x}(\partial\Omega\setminus E_Q)\lesssim\ve^\alpha\quad\mbox{ for all $x\in 2V_Q$.}
\end{equation}
Therefore, 
$$
\bigl|g_Q(x) - u_Q(x)\bigr|\lesssim \ve^\alpha\quad\mbox{ for all $x\in 2V_Q$.}
$$
Since $g_Q-u_Q$ is harmonic, we have
\begin{equation}\label{eqggg*2}
\bigl|\nabla(g_Q - u_Q)(x)\bigr|\lesssim \frac1{r(V_Q)}\;\avint_{2V_Q}|g_Q-u_Q|\,dy\lesssim
\ve^\alpha\,\frac{1}{r(V_Q)}\quad\mbox{ for all $x\in V_Q$,}
\end{equation}
and so, assuming $\ve$ small enough, 
$$-\partial_1u_Q(x) \geq -\partial_1g_Q(x) - \bigl|\nabla(g_Q - u_Q)(x)\bigr| \gtrsim \frac{1}{r(V_Q)} \quad\mbox{ for all $x\in V_Q$,}$$
which concludes the proof of \rf{eqkey29}.

\vv
The final estimate \rf{eqkey30} is an immediate consequence of \rf{eqkey29}.
\end{proof}

\vv

\begin{proof}[Proof of Lemma \ref{lemapprox} in the case $n=1$]
As above, we assume that $p_Q-y_Q$ is parallel to the $x$ axis and that $p_{Q,1}>y_{Q,1}$,
so that
\begin{equation}\label{eqfh5890}
0<\frac{x_1-y_1}{|x-y|^2} \approx \frac1{r(V_Q)}\quad\mbox{for all $x\in V_Q$ and all $y\in B(y_Q,r(V_Q))$.}
\end{equation}

We now define a function $g_Q$ which will play the role of the analogous one in \rf{deggg*}. To this end, note that because of the AD-regularity of $\mu$ there exists some point 
$$y_2\in\supp\mu \cap A\bigl(y_Q,3\ell(Q), 4c_0^2\ell(Q)\bigr),$$
where $A(x,r_1,r_2)$ stands for the open annulus centered at $x$ with inner radius $r_1$ and
outer radius $r_2$, and 
$c_0$ is the AD-regularity constant of $\mu$.
Consider a ball $B_2$ centered at $y_2$ with radius $r(V_Q)$. To shorten notation we write
$B_1=B(y_Q,r(V(Q))$, $r=r(V_Q)$, and $y_1=y_Q$. Then we define
\begin{equation}\label{deggg***}
g_Q(x) =  \int_{B_1} \frac1r\,\log\frac1{|x-y|}\,d\mu(y) - \frac{\mu(B_1)}{\mu(B_2)}
\int_{B_2} \frac1r\,\log\frac1{|x-y|}\,d\mu(y).
\end{equation}
We claim that $g_Q$ satisfies the following properties:
\vv
\begin{itemize}
\item[(a)] $g_Q\in C_0(\overline \Omega)$,
\vv
\item[(b)] $g_Q\geq0$ on $Q$ and $\|g_Q\|_\infty\lesssim |\log\ve|$,
\vv
\item[(c)] $ 0\leq -\partial_1 g_Q(x)  \approx \dfrac1{r(V_Q)}$
 for all $x\in V_Q$.
\end{itemize}

Using the properties above and arguing as in the case $n\geq2$, we can complete the proof of the lemma
in this case. Indeed, from \rf{eqggg*0}, taking account that  
$
\omega^{p_Q}(\partial\Omega\setminus E_Q)  \lesssim \ve^\alpha$ (because \rf{eqggg*1} is still valid)
and that $\|g_Q\|_\infty\lesssim |\log\ve|$, we deduce that
$$\bigl|g_Q(x) - u_Q(x)\bigr|\leq \ve^\alpha\,|\log\ve|\quad \mbox{for all $x\in V_Q$.}$$
Then, as in \rf{eqggg*2} we derive that
$$\bigl|\nabla(g_Q - u_Q)(x)\bigr|\lesssim 
\ve^\alpha\,|\log\ve|\,\frac{1}{r(V_Q)}\quad\mbox{ for all $x\in V_Q$,}$$
which together with the property (c) above yields \rf{eqkey29}, for $\ve$ small enough. 

Again, the final estimate \rf{eqkey30} is an immediate consequence of \rf{eqkey29}.

\vv
We now verify the claims (a), (b) and (c). The continuity of $g_Q$ follows easily from the local
$\mu$ uniform integrability of the kernel $\log\frac1{|x-y|}$. To see that it vanishes at infinity, note that
$g_Q$ can be written as follows:
\begin{equation}\label{eqgp94}
g_Q(x) =  \int_{B_1} \frac1r\,\log\frac{|x-y_1|}{|x-y|}\,d\mu(y) - \frac{\mu(B_1)}{\mu(B_2)}
\int_{B_2} \frac1r\,\log\frac{{|x-y_1|}}{|x-y|}\,d\mu(y),
\end{equation}
and then 
$$\log\frac{|x-y_1|}{|x-y|}\to 0 \quad \mbox{as $x\to\infty$.}$$

To show that $g_Q\geq0$ on $Q$, write
$$g_Q(x) =  \int_{B_1} \frac1r\,\log\frac{\ell(Q)}{|x-y|}\,d\mu(y) - \frac{\mu(B_1)}{\mu(B_2)}
\int_{B_2} \frac1r\,\log\frac{{\ell(Q)}}{|x-y|}\,d\mu(y) =: g_1(x)- g_2(x).
$$
Observe that
$$\frac{\ell(Q)}{|x-y|}> 1\quad \mbox{ for all $y\in B_1$,}$$
while
$$\frac{\ell(Q)}{|x-y|}< 1\quad \mbox{ for all $y\in B_2$.}$$
So $g_1(x)>0$ and $g_2(x)<0$ for $x\in Q$, and thus $g_Q(x)>0$ on $Q$.

To estimate $\|g_Q\|_\infty$, suppose first that $x\not\in B(y_1,10c_0^2\ell(Q))$. For these 
points $x$ we have
$$|x-y|\approx |x-y_1|\approx |x-y_2|\quad \mbox{ for all $y\in B_1\cup B_2$.}$$
So 
$$-C\leq \log\frac{|x-y_1|}{|x-y|}\leq C \quad \mbox{for $x\not\in B(y_1,10c_0^2\ell(Q))$ and $y\in B_1\cup B_2$.}.$$
Then, from the identity \rf{eqgp94}, taking into account that $\mu(B_1)\approx\mu(B_2)$ we deduce
$$|g_Q(x)|\lesssim \frac{\mu(B_1)}r + \frac{\mu(B_2)}r \lesssim 1 \quad
\mbox{for $x\not\in B(y_1,10c_0^2\ell(Q))$}.$$
In the case $x\in B(y_1,10c_0^2\ell(Q))$ we write
$$g_Q(x) =  \int_{B_1} \frac1r\,\log\frac{r}{|x-y|}\,d\mu(y) - \frac{\mu(B_1)}{\mu(B_2)}
\int_{B_2} \frac1r\,\log\frac{r}{|x-y|}\,d\mu(y) =:\wt g_1(x)- \wt g_2(x).
$$
Let us estimate $\wt g_1(x)$. To this end, note first that if $x\in B(y_1,10c_0^2\ell(Q))\setminus 2B_1$, then
$$1\leq\frac{|x-y|}r \lesssim\frac{\ell(Q)}r \approx \frac1\ve.$$
Hence $|\log\frac{r}{|x-y|}|\lesssim |\log\ve|$ and thus $|\wt g_1(x)|\lesssim |\log\ve|$. On the other hand,
if $x\in 2B_1$, then
$$|\wt g_1(x)|\leq \int_{B(x,4r)} \frac1r\,\left|\log\frac{r}{|x-y|}\right|\,d\mu(y).$$
By the linear growth of $\mu$ it easy to check that the last integral is bounded above by
some constant depending only on the growth constant for $\mu$. So in any case we have $|\wt g_1(x)|\lesssim |\log\ve|$ for $x\in B(y_1,10c_0^2\ell(Q))$. The same estimate holds for $\wt g_2(x)$, and then it follows 
that 
$$|g_Q(x)|\lesssim |\log\ve|\quad \mbox{ for all $x\in B(y_1,10c_0^2\ell(Q))$,}$$
which concludes the proof of $\|g_Q\|_\infty\lesssim |\log\ve|$.

Finally we verify (c).
We have
$$-\partial_1 g_Q(x) =  \int_{B_1} \frac1{2r}\,\frac{x_1-y_1}{|x-y|^2}\,d\mu(y) - \frac{\mu(B_1)}{\mu(B_2)}
\int_{B_2} \frac1{2r}\,\frac{x_1-y_1}{|x-y|^2}\,d\mu(y) =: h_1(x) - h_2(x).$$ 
From \rf{eqfh5890} we get
$$0<h_1(x) \approx \frac1{r(V_Q)} \quad\mbox{ for all $x\in V_Q$,}$$
while taking into account that $\dist(V_Q,B_2)\approx\ell(Q)$, we have
$$|h_2(x)| \lesssim \frac1{\ell(Q)} \quad\mbox{ for all $x\in V_Q$,}$$
Thus
$$-\partial_1 g_Q(x) \gtrsim \frac1{r(V_Q)} - \frac c{\ell(Q)} \approx \frac1{r(V_Q)},$$
and so the proof of the claim is concluded.
\end{proof}

\vv


\subsection{The stopping cubes}

Before defining the family $\ttt$, we need to define,
for any given $R\in\DD_\mu$, two associated families
$\HD(R)$ and $\LD(R)$ of high density and low density cubes, respectively.

Let $0<\delta\ll1$ and $A\gg1$ be some fixed constants. For a fixed a cube $R\in\DD_\mu$,
let $Q\in\DD_\mu$, $Q\subset R$.
We say that $Q\in\HD(R)$ (high density) if $Q$ is a maximal cube satisfying
$$\frac{\omega^{p_R}(2Q)}{\mu(2Q)}\geq A \,\frac{\omega^{p_R}(2R)}{\mu(2R)}.$$
We say that $Q\in\LD(R)$ (low density) if $Q$ is a maximal cube satisfying
$$\frac{\omega^{p_R}(Q)}{\mu(Q)}\leq \delta \,\frac{\omega^{p_R}(R)}{\mu(R)}$$
(notice that $\omega^{p_R}(R)\approx \omega^{p_R}(2R)\approx 1$ by \rf{Bour}). Observe that the definition of the family $\HD(R)$ involves the density of $2Q$, while the one of $\LD(R)$ involves the density of $Q$.

We denote
$$B_H(R)=\bigcup_{Q\in \HD(R)} Q \quad\mbox{ and }\quad B_L(R)=\bigcup_{Q\in \LD(R)} Q.$$
\vv

\begin{lemma}\label{lemhd}
We have
$$\mu(B_H(R)) \lesssim \frac1A\,\mu(R).$$
\end{lemma}

\begin{proof}
By Vitali's covering theorem, there exists a subfamily $I\subset\HD(R)$ so that the cubes $2Q$, $Q\in I$, are pairwise disjoint and 
$$\bigcup_{Q\in\HD(R)} 2Q\subset \bigcup_{Q\in I} 6Q.$$
Then, using that $\mu$ is doubling,
\begin{align*}
\mu(B_H(R)) \lesssim \sum_{Q\in I}\mu(2Q) \leq \frac 1A
\sum_{Q\in I}\frac{\omega^{p_R}(2Q)}{\omega^{p_R}(2R)}\,\mu(2R) \lesssim
\frac1A\,\mu(R).
\end{align*}
\end{proof}

\vv
Concerning the low density cubes, we have:

\begin{lemma}\label{lemld}
We have
$$\omega^{p_R}(B_L(R)) \leq \delta\,\omega^{p_R}(R).$$
\end{lemma}

\begin{proof}
Since the cubes from $\LD(R)$ are pairwise disjoint, we have
\begin{align*}
\omega^{p_R}(B_L(R)) = \sum_{Q\in\LD(R)}\omega^{p_R}(Q) \leq \delta
\sum_{Q\in\LD(R)}\frac{\mu(Q)}{\mu(R)}\,\omega^{p_R}(R) \leq \delta\,\omega^{p_R}(R). 
\end{align*}
\end{proof}
\vv


\subsection{The family $\ttt(R_0)$ and the trees of the corona decomposition}\label{sectop}

In this subsection, we define, for each $R_0\in\DD_\mu$,  a localized version of the family $\ttt$, which
we will denote by $\ttt(R_0)$.
To this end, given a cube $R\in\DD_\mu$ we let 
$$\sss(R):=\{ S \in \HD(R)\cup\LD(R):  \nexists \,\,\wt S \in \HD(R)\cup\LD(R) \,\, \textup{such that}\,\,  S \subsetneq \wt S\}.$$
Notice that by maximality with respect to the inclusion in $\HD(R)\cup\LD(R)$, $\sss(R)$ is a family of pairwise disjoint cubes. We define 
$$\tree(R):=\{ Q \in \DD_\mu (R): \nexists \,\, S \in \sss(R) \,\, \textup{such that}\,\, Q\subsetneq S\}.$$ 
 In particular, note that $\sss(R)\subset\tree(R)$. We also define
 $$\wt\sss(R):= \{Q \in \DD_\mu (R): \exists \,\, S \in \sss(R) \,\, \textup{such that}\,\, Q\in \ch(S)\},$$ 
 where $\ch(S)$ stands for the children of $S$. Notice that this family is also pairwise disjoint.

We fix a cube $R_0\in\DD_\mu$ and we define the family of the top cubes with respect to $R_0$ as follows:
first we define the families $\ttt_k(R_0)$ for $k\geq0$ inductively. We set
$$\ttt_0(R_0)=\{R_0\}.$$
Assuming that $\ttt_k(R_0)$ has been defined, we set
$$\ttt_{k+1}(R_0) = \bigcup_{R\in\ttt_k(R_0)}\wt\sss(R),$$
and then we define
$$\ttt(R_0)=\bigcup_{k\geq0}\ttt_k(R_0).$$
Notice that
$$\DD_\mu(R_0)= \bigcup_{R\in\ttt(R_0)}\tree(R),$$
and this union is disjoint.

We denote by $\ttt_H(R_0)$ the subfamily of the cubes from $\ttt(R_0)$ whose parents belong to $\HD(R)$ for some $R\in\ttt(R_0)$, and by $\ttt_L(R_0)$ the subfamily of the cubes from $\ttt(R_0)$ whose parents belong to $\LD(R)$ for some $R\in\ttt(R_0)$. So we have
$$\ttt(R_0)=\{R_0\}\cup\ttt_H(R_0)\cup\ttt_L(R_0).$$
Observe also that if $Q\in\ttt_H(R_0)$ (resp.\ $\ttt_L(R_0)$), then any sibling of $R$ also belongs to
$\ttt_H(R_0)$ (resp.\ $\ttt_L(R_0)$).
\vv

\begin{lemma}\label{lemfac84}
For any $R\in\ttt(R_0)$, the following hods:
$$
\omega^{p_R}(3Q)\approx_{\delta,A} \frac{\mu(Q)}{\mu(R)}\quad\mbox{ for all $Q\in\tree(R)$.}$$
\end{lemma}

\begin{proof}
Let $\wh Q\subset \DD_\mu$ be the parent  of $Q$.
It is immediate to check that
$\wh Q\subset 3Q\subset 2 \wh Q$. If $\wh Q\subset R$, by construction
$$\omega^{p_R}(3Q)\leq \omega^{p_R}(2\wh Q) \leq A\,\omega^{p_R}(2R)\,\frac{\mu(2Q)}{\mu(2R)} \approx A\,\frac{\mu(Q)}{\mu(R)},
$$
and also
$$\omega^{p_R}(3Q)\geq \omega^{p_R}(\wh Q) \geq \delta\,\omega^{p_R}(R)\,\frac{\mu(Q)}{\mu(R)} \approx \delta\,\frac{\mu(Q)}{\mu(R)}.$$

If $\wh Q\not\subset R$, then $\wh Q$ is the parent of $R$ and thus $3Q\supset R$, which implies that 
$\omega^{p_R}(3Q)\approx 1$ and 
$\mu(Q)\approx\mu(R)$. Hence the estimate in the lemma is trivially true.
\end{proof}

\vv


\subsection{The iterative construction and the key lemma}

Our next goal is to prove  that the family $\ttt(R_0)$ satisfies a packing condition analogous to the one 
stated in \rf{eqpack57} for the family $\ttt$.
The proof would be easy if the inequality $\mu(B_L(R))) \ll \mu(R)$ followed from Lemma 3.5, but   
we are unable to verify that. 
 In this subsection we instead prove a
variant of the above inequality for $B_L^m(R)$ for some $m\geq 1$. The set $B_L^m(R)$ is defined as follows.

For $R\in\DD_\mu$, we denote $\LD^0(R)=\{R\}$, $\LD^1(R)=\LD(R)$, and for $k\geq1$ we consider the families of cubes
$$\LD^{k+1}(R) = \bigcup_{Q\in\LD^k(R)} \LD(Q),$$
and the subset of $R$ given by
$$B_L^k(R)=\bigcup_{Q\in\LD^k(R)} Q.$$
Notice that the stopping conditions in the definition of the family of low density cubes $\LD^k(R)$
involve the harmonic measure $\omega^{p_Q}$ for a suitable $Q\in\LD^{k-1}(R)$, instead of $\omega^{p_R}$.

\vv
The next lemma is one of the key steps for the proof of Theorem \ref{teo1}.

\begin{lemma}[Key Lemma]\label{lemclau}
Suppose that either the assumption (b) or (c) in Theorem \ref{teo1} holds. Suppose also that $\ve$ in \rf{eq1} is chosen small enough
and that $\delta\leq\ve$. Then for any $m\ge1$ we have
\begin{equation}\label{eqg**p45}
\sum_{k=1}^m \,\sum_{Q\in\LD^k(R)}\mu(Q)\lesssim_\ve\mu(R)
\end{equation}
and
\begin{equation}\label{eqg**p46}
\mu(B_L^m(R))\lesssim_\ve\frac1m\,\mu(R).
\end{equation}
\end{lemma}

\begin{proof}
For $Q\subset\DD_\mu$, $Q\subset R$, we denote 
$$E_Q = Q\setminus B_L(Q).$$
By Lemma \ref{lempq0} and Lemma \ref{lemld} applied to $Q$,
\begin{align}\label{eqhm11}
\omega^{p_Q}(E_Q) & = \omega^{p_Q}(Q) - \omega^{p_Q}(B_L(Q)) \geq\\
\geq (1-\delta)\,\omega^{p_Q}(Q) &\geq (1-\delta)(1-c\,\ve^\alpha)\geq 1-c'\ve^\alpha.\nonumber
\end{align}
Hence, by Lemma \ref{lemapprox}, if $\ve$ is small enough and $\delta\leq\ve$, there exists a function $u_Q$ 
 on $\Omega$ and a non-negative Borel function $f_Q$ with\footnote{In the case $n\geq2$ this can be improved to $f_Q\leq c\,\chi_{E_Q}$, but we will not use this.} $f_Q\leq c\,|\log\ve|\,\chi_{E_Q}$
 such that
$$u_Q(x) = \int_{E_Q} f_Q\,d\omega^x$$
satisfying, for some unit vector $e_Q\in\R^{n+1}$, 
$$\nabla u_Q(x)\cdot e_Q\geq c\,\frac1{r(V_Q)}\quad\mbox{ for all $x\in V_Q$,}$$
and so that
$$\int_{V_Q} |\nabla u_Q(x)|^2\,\dist(x,\partial\Omega)\,dx\geq c\,r(V_Q)^n.$$

Notice that the set $E_Q$ is disjoint form the low density cubes from $\LD(Q)$, so that  by construction
the sets $E_Q$,  $Q\in\LD^k(R)$, $k\geq1$, are pairwise disjoint.
This implies that the function
$$u := \sum_{k=1}^m \,\sum_{Q\in\LD^k(R)} u_Q$$
is uniformly bounded by $c\,|\log\ve|$ on $\Omega$. Indeed, by the definitions of the functions $u$ and $u_Q$,
\begin{align}\label{eqhm12}
u(x) = \int \sum_{k=1}^m \,\sum_{Q\in\LD^k(R)} f_Q\,\chi_{E_Q}\,d\omega^x & \\
\leq c\,|\log\ve|\sum_{k=1}^m \,\sum_{Q\in\LD^k(R)}\hm^{x}(E_Q) &\leq  c\,|\log\ve|.\nonumber
\end{align}
Remark that the latter estimate also holds with $u$ replaced by $u-u_Q$. 

We claim that for all $x\in V_Q$, for $Q\in\LD^k(R)$, $k\geq1$,
\begin{equation}\label{eq22}
|\nabla u(x)|\geq \nabla u(x)\cdot e_Q\gtrsim \frac1{r(V_Q)}.
\end{equation}
To show this, we set
$$\nabla u(x)\cdot e_Q\geq \nabla u_Q(x)\cdot e_Q - |\nabla(u-u_Q)(x)| \geq \frac{c}{r(V_Q)} - |\nabla(u-u_Q)(x)|.$$
Since $u-u_Q$ is harmonic and positive in $3V_Q$, we have
$$|\nabla(u-u_Q)(x)| \lesssim \frac1{r(V_Q)}\|u-u_Q\|_{L^\infty(2V_Q)}\approx
\frac1{r(V_Q)}(u-u_Q)(p_Q).$$
Now, since $u-u_Q$ is harmonic in $\Omega$ and vanishes in $ E_Q$ we obtain
$$(u-u_Q)(p_Q) = \int (u-u_Q)\,d\omega^{p_Q} \leq \|u-u_Q\|_{L^\infty(\partial\Omega)}\,\omega^{p_Q}(\partial\Omega\setminus E_Q) \leq C\,|\log\ve|\,\ve^\alpha,$$
where in the last inequality we used \rf{eqhm11} along with \rf{eqhm12} for $u-u_Q$.
Hence
$$\nabla u(x)\cdot e_Q\geq \frac{c}{r(V_Q)} - \frac{C\,|\log\ve|\,\ve^\alpha}{r(V_Q)}$$
and our claim follows if $\ve$ is taken small enough.
\vv

$\bullet$ Suppose first that the assumption (c) in Theorem \ref{teo1} holds.
From the claim \rf{eq22} and the fact that the sets $V_Q$, $Q\in\LD^k(R)$, $k\geq1$, are pairwise disjoint (or at least, have bounded overlap), we get
\begin{align}\label{eqalof25}
\int_{B(R)} |\nabla u|^2\,\dist(x,\partial\Omega)\,dx & \gtrsim \sum_{k=1}^m \,\sum_{Q\in\LD^k(R)}
\int_{V_Q} |\nabla u(x)|^2\,\dist(x,\partial\Omega)\,dx\\
& \gtrsim \sum_{k=1}^m \,\sum_{Q\in\LD^k(R)}r(V_Q)^n
\approx_\ve \sum_{k=1}^m \,\sum_{Q\in\LD^k(R)}\mu(Q),\nonumber
\end{align}
where $B(R)$ is some big ball concentric with $R$, with radius comparable to $\ell(R)$, which contains
the sets $V_Q$, $Q\in\LD^k(R)$, $k=1,\ldots,m$.
Then, from \rf{eqhip1} we derive
$$\sum_{k=1}^m \,\sum_{Q\in\LD^k(R)}\mu(Q)\lesssim_\ve\mu(R),$$
which yields the first assertion of the lemma in this case.

\vv
$\bullet$ Suppose now that the hypothesis (b) in Theorem \ref{teo1} holds, i.e., that for all $\varepsilon_0  > 0$
every bounded harmonic function on $\Omega$ is $\varepsilon_0$-approximable. 
So, for some $\ve_0>0$ small enough to be chosen below, let 
 $\varphi \in W^{1,1}_{\rm {loc}}(\Omega)$ such that
$\|u - \varphi\|_{L^{\infty}(\Omega)} < \varepsilon_0$ and 
\begin{equation}\label{eqcc23}
\int_{B(R)} |\nabla \varphi(y)|\, dy \leq C\,\mu(R),
\end{equation}
where $B(R)$ is as above.
We claim that
\begin{equation}\label{claimcla}
\int_{V_Q}|\nabla\vphi(y)|\,dy\gtrsim_\ve\ell(Q)^n\quad\mbox{ for all $Q\in\LD^k(R)$, $k=1,\ldots,m$.}
\end{equation}
Indeed,  for each such $Q$ consider two balls $V_Q^1,V_Q^2\subset V_Q$ such that $r(V_Q^1) = r(V_Q^1) =\frac1{100} r(V_Q)$, and so that $V^2_Q=\frac{r(V_Q)}{10}e_Q+V^1_Q$ (i.e., $V_Q^2$ is the translation of $V_Q^1$ by the vector $\frac{r(V_Q)}{10}e_Q$). Then, by a change of variable, the mean value theorem, and \rf{eq22} it follows  that 
\begin{align*}
\avint_{V^2_Q} u(y)\,dy - \;\avint_{V^1_Q} u(y)\,dy &= \;\avint_{V^1_Q} \bigl( u(y+\tfrac {r(V_Q)}{10}e_Q) - u(y)\bigr)
\,dy\\
&\geq c \,r(V_Q)\,\min_{y\in V_Q} \bigl[\nabla u(y)\cdot e_Q\bigr]\gtrsim 1.
\end{align*}
Hence, if $\ve_0$ is small enough, then we also have
$$\avint_{V^2_Q} \vphi(y)\,dy - \;\avint_{V^1_Q} \vphi(y)\,dy \gtrsim 1.$$
Then  \rf{claimcla} is an immediate consequence of the Poincar\'e inequality applied to the ball $V_Q$.

Arguing as in \rf{eqalof25}, from   
\rf{claimcla} and \rf{eqcc23} we deduce
$$
\sum_{k=1}^m \,\sum_{Q\in\LD^k(R)}\mu(Q) \lesssim_\ve
\int_{B(R)} |\nabla \varphi(y)| dy\lesssim \mu(R),$$
which completes the proof of the first assertion of the lemma.
\vvv

The second estimate in the lemma follows from the fact that if $Q\in B_L^m(R)$, then $x$ belongs to $m$ different cubes $Q\in\LD^k(R)$, $k=1,\ldots,m$, so that
$$\sum_{k=1}^m \,\sum_{Q\in\LD^k(R)}\chi_Q(x) = m,$$
and by Chebyshev,
$$\mu(B_L^m(R))\leq \frac1m\,\sum_{k=1}^m \,\sum_{Q\in\LD^k(R)}\mu(Q)\lesssim_\ve \frac1m\,\mu(R).$$
\end{proof}
\vv

\begin{rem}\label{rem***}
The preceding lemma also holds if we assume that either the assumption (b) or (c) in Theorem \ref{teo1} is 
only satisfied  by functions $u \in C(\overline \Omega)$ which are bounded and harmonic in $\Omega$. 
The proof is almost the same. Indeed, consider a finite subfamily $\LL\subset \bigcup_{k=0}^m\LD^k(R)$.
For each $Q\in\LL$, set
$$\wt E_Q = \tfrac34B_Q\setminus \bigcup_{P\in\LL:P\subsetneq Q} P.$$
Arguing as in \rf{eqhm11}, we deduce that
$\omega^{p_Q}(\wt E_Q)  \geq 1-c\,\ve^\alpha$.
Consider a continuous function $\vphi_Q$ such that $0\leq \vphi_Q\leq 1$ with
$$\vphi_Q(x)=\left\{\begin{array}{ll}
1& \quad\mbox{if $x\in \wt E_Q$,}\\
\\
0& \quad\mbox{if $x\in (\frac45 B_Q)^c \cup \bigcup_{P\in\LL:P\subsetneq Q} \frac45B_P$.}
\end{array}\right.
$$
Set
$$\wt u_Q = \int \vphi_Q\,g_Q\,d\omega^x,$$
where $g_Q$ is defined in \rf{deggg*} and \rf{deggg***}. Since both $\vphi_Q$ and $g_Q$ are continuous in $\overline\Omega$, it follows easily that
$\wt u_Q\in C_0(\overline\Omega)$. Consider the function
$\wt u:=\sum_{Q\in\LL} \wt u_Q.$
This satisfies
$$
\wt u(x) = \int \sum_{Q\in\LL} \vphi_Q\,g_Q\,d\omega^x
\leq c\,|\log\ve|\sum_{Q\in\LL} \int  \vphi_Q\,d\omega^x\leq  c\,|\log\ve|,
$$
because $\sum_{Q\in\LL}   \vphi_Q|_{\partial\Omega}\leq 1$ by construction. The same arguments in the proof of Lemma \ref{lemclau}, with $u$ replaced by
$\wt u$, show then that
$$\sum_{Q\in\LL}\mu(Q)\lesssim_\ve\mu(R),$$
with the implicit constant independent of $\#\LL$. So \rf{eqg**p45} holds and, also, \rf{eqg**p46}.
\end{rem}
\vv


\subsection{The packing condition of the family $\ttt(R_0)$}

\begin{lemma}\label{lemcarleson}
There exists a constant $C$ such that
for any $Q_0\in\DD_\mu(R_0)$,
\begin{equation}\label{eqpack00}
\sum_{R\in\ttt(R_0):R\subset Q_0} \mu(R)\leq C\,\mu(Q_0).
\end{equation}
\end{lemma}

\begin{proof}
First note that it is enough to prove the lemma assuming that $Q_0\in\ttt(R_0)$. Indeed, given any arbitrary $Q_0\in\DD_\mu(R_0)$,
we consider the family $\cM$ of maximal cubes from $\ttt(R_0)\cap \DD_\mu(R_0)$, and  apply \rf{eqpack00} to each $S\in\cM$ to
obtain
$$\sum_{R\in\ttt(R_0):R\subset Q_0} \mu(R)= \sum_{S\in\cM} \,\sum_{R\in\ttt(R_0):R\subset S} \mu(R) \leq C\,\sum_{S\in\cM}\mu(S)\leq C\,\mu(R).$$

Therefore we assume that $Q_0\in\ttt(R_0)$. We denote
$$\wt\ttt_H(Q_0) := \DD_\mu(Q_0)\cap\ttt_H(R_0) \qquad
\mbox{and}\qquad \wt \ttt_L(Q_0) := \DD_\mu(Q_0)\cap\ttt_L(R_0).$$
We split $\DD_\mu(Q_0)$ into trees whose roots are all the cubes from 
$\{Q_0\}\cup \wt \ttt_H(Q_0)$.
That is, for each $R\in\{Q_0\}\cup \wt \ttt_H(Q_0)$, we consider the tree $\wt\tree(R)$ formed by the
cubes from $\DD_\mu(R)$ which are not contained in any other cube from $\wt \ttt_H(Q_0)\cap\DD_\mu(R)$ different from $R$. So we have the partition
\begin{equation}\label{eqpart}
\DD_\mu(Q_0) = \bigcup_{R\in\{Q_0\}\cup\wt \ttt_H(Q_0)}\wt\tree(R).
\end{equation}
Also, we denote by $\nex(R)$ the family of the maximal cubes from $\wt \ttt_H(Q_0)\cap\DD_\mu(R)$ different from $R$.

By construction, for each $R\in\{Q_0\}\cup\wt\ttt_H(Q_0)$ (taking into account that $Q_0\in\ttt(R_0)$) we have
$$\nex(R)\subset \bigcup_{k\geq0}\,\bigcup_{Q\in\LD^k(R)}\HD(Q).$$
Then, by Lemmas  \ref{lemhd} and \ref{lemclau},
$$\sum_{P\in\nex(R)}\mu(P)\leq \sum_{k\geq0}\,\sum_{Q\in\LD^k(R)}\sum_{S\in\HD(Q)}\mu(S)
\leq\frac CA\,\sum_{k\geq0}\,\sum_{Q\in\LD^k(R)}\mu(Q) \leq \frac CA\,\mu(R).$$
So, assuming $A$ big enough, we have
$$\sum_{P\in\nex(R)}\mu(P)\leq  \frac12\,\mu(R),$$
which is equivalent to saying that
$$\mu\biggl(R\setminus\bigcup_{P\in\nex(R)}P\biggr)\geq  \frac12\,\mu(R).$$
Since the sets $R\setminus\bigcup_{P\in\nex(R)}P$, with $R\in\wt\ttt_H(Q_0)$ are pairwise disjoint, we 
obtain
\begin{equation}\label{eqtoph}
\sum_{R\in\wt\ttt_H(Q_0)} \mu(R) \leq 2 \sum_{R\in\wt\ttt_H(Q_0)}\mu\biggl(R\setminus\bigcup_{P\in\nex(R)}P\biggr) \leq 2\,\mu(Q_0).
\end{equation}

Now it remains to bound the sum $\sum_{R\in\wt\ttt_L(Q_0)} \mu(R)$.
In view of \rf{eqpart} we can split this sum as follows:
$$\sum_{R\in\wt\ttt_L(Q_0)} \mu(R) = \sum_{S\in\{Q_0\}\cup\wt\ttt_H(Q_0)} \sum_{R\in\wt\tree(S)\cap\wt\ttt_L(Q_0)}
\mu(R).$$
By construction, since $Q_0\in \ttt(R_0)$, for each $S\in\{Q_0\}\cup \wt\ttt_H(Q_0)$  we have
$$\sum_{R\in\wt\tree(S)\cap\wt\ttt_L(Q_0)}\mu(R) \leq\sum_{k\geq0}\,\sum_{Q\in\LD^k(S)}\mu(Q),$$
which does not exceed $C\,\mu(S)$ by Lemma \ref{lemclau} again,  so  that
$$\sum_{R\in\wt\ttt_L(Q_0)} \mu(R)\leq C \sum_{S\in\{Q_0\}\cup\wt\ttt_H(Q_0)}\mu(S)\leq C\,\mu(Q_0),$$
by \rf{eqtoph}. This completes the proof of the lemma.
\end{proof}

\vv


\subsection{Conclusion of the proof of Proposition \ref{propo11}}

If $\diam(\partial\Omega)<\infty$, then we choose $R_0=\partial\Omega$ and we define
$\ttt = \ttt(R_0)$.
By Lemmas \ref{lemfac84} and \ref{lemcarleson}, the family $\ttt$ satisfies the properties required in Proposition 
\ref{propo11}.

\vv In the case when $\partial\Omega$ is not bounded we apply a technique described in p.\ 38 of \cite{DS1}: we consider a family of cubes $\{R_j\}_{j\in J}\in\DD_\mu$ which are pairwise disjoint, whose union is all of $\supp\mu$, and which have the property 
that for each $k$ there at most $C$ cubes from $\DD_{\mu,k}$ not contained in any cube $R_j$. For each
$R_j$ we construct a family $\ttt(R_j)$ analogous to $\ttt(R_0)$.
Then we set 
$$\ttt=\bigcup_{j\in J} \ttt(R_j) \cup \BZ,$$
where $\BB\subset\DD_\mu$ is the family of cubes which are not contained in any cube $R_j$, $j\in J$. 
One can easily check that the family $\ttt$ satisfies all the properties from Proposition \ref{propo11}. See p.\ 38 of \cite{DS1}
for the construction of the family $\{R_j\}$ and additional details.

\vv


\section{Proof of Theorem \ref{teo1} for bounded uniform domains}\label{sec5}

In the rest of the paper we allow all the constants $C$ and other implicit constants to depend on the parameter $\ve$
from Subsection \ref{sec:aproximation}.

In this section we will complete the proof of Theorem \ref{teo1} in the special case when $\Omega$ is a bounded uniform domain.
For this type of domain the proof is simpler and more transparent than in the general case  because the Harnack chain condition holds in a uniform domain and for that reason
 we think  it is useful to  first give the proof in this special case.  
If a uniform domain has  
 AD-regular boundary (or more generally, it satisfies the CDC), then  by the Harnack chain condition 
 the harmonic measure  $\omega^{p_Q}$ in $\Omega$ is doubling, with the doubling constant bounded above independently of the pole $p_Q$ (see \cite{AiHi}). 
Then by the theorem of Hofmann, Martell and Uriarte-Tuero in \cite{HMU}, the uniform rectifiability of $\partial\Omega$ is equivalent
to the $A_\infty(\mu)$ property of $\omega^p$, for a fixed $p\in\Omega$, and we will use this criterion to prove 
Theorem \ref{teo1} in the case of uniform domains.

\subsection{The set $R\setminus B_L^m(R)$}

 In the rest of this section we assume that  $\Omega$ is a bounded uniform domain with $n$-AD-regular boundary.

For a fixed $R\in\DD_\mu$, we choose $m$ big enough so that
\begin{equation}\label{eqak941}
\mu(R\setminus B_L^m(R)) = \mu(R) -  \mu(B_L^m(R))
\geq \mu(R) - \frac Cm \,\mu(R)\geq \frac12\,\mu(R),
\end{equation}
by applying Lemma \ref{lemclau}.
\vv

\begin{lemma}\label{lemg}
There is a function $g\in L^\infty(\omega^{p_R}|_{R\setminus B_L^m(R)})$ such that
$$\frac1{\mu(R)}\,\mu|_{R\setminus B_L^m(R)} = g\,\omega^{p_R}|_{R\setminus B_L^m(R)},$$
with 
$$\|g\|_{L^\infty(\omega^{p_R}|_{R\setminus B_L^m(R)})} \leq C(\delta,m).$$
\end{lemma}

\begin{proof}
By the Lebesgue differentiation theorem, it is enough to show that, given any $x\in R\setminus B_L^m(R)$, for any $Q\in\DD_\mu$ such that $x\in Q\subset R$, with $\ell(Q)$ small enough,
$$\mu(Q) \leq C(\delta,m) \,\omega^{p_R}(Q)\,\mu(R).$$

For such point $x$, there exists $0\leq j\leq m-1$ and some cube $Q_j\in\LD^j(R)$ such that
$$x\in Q_j\quad\mbox{ but }\quad x\not\in\bigcup_{P\in\LD(Q_j)} P.$$
Consider now the cubes $R=Q_0,Q_1,\ldots,Q_j$ such that $x\in Q_k\in\LD^k(R)$ for $0\leq k\leq j$.
To simplify notation, we denote $p_k=p_{Q_k}$. For any $Q\in\DD_\mu$
such that $Q\subset Q_j$, we have
\begin{equation}\label{eqcad22}
\frac{\omega^{p_j}(Q)}{\mu(Q)}\geq \delta \,\frac{\omega^{p_j}(Q_j)}{\mu(Q_j)} \approx \frac{\delta}{\mu(Q_j)}.
\end{equation}
Analogously, by the definition of $\LD(Q_k)$ and the doubling property of $\omega^{p_k}$,
\begin{equation}\label{eqcad23}
\frac{\omega^{p_k}(Q_{k+1})}{\mu(Q_{k+1})}\approx \delta \,\frac{\omega^{p_k}(Q_k)}{\mu(Q_k)}
\approx \frac{\delta}{\mu(Q_k)}.
\end{equation}

Notice that for any $k$, since $\Omega$ is a uniform domain with AD-regular boundary,
$$\omega^{p_R}(Q) \approx \omega^{p_{j}}(Q)\,\omega^{p_R}(Q_{j}),$$
and also
$$\omega^{p_R}(Q_k) \approx \omega^{p_{k-1}}(Q_k)\,\omega^{p_R}(Q_{k-1}).$$
Therefore,
$$\omega^{p_R}(Q) \approx \omega^{p_{j}}(Q)\,\omega^{p_{j-1}}(Q_j)\,\omega^{p_{j-2}}(Q_{j-1})\ldots\omega^{p_{R}}(Q_1),$$
where the implicit constant is of the form $C^j\leq C^m$.
Plugging here the estimates \rf{eqcad22} and \rf{eqcad23} we get
$$\omega^{p_R}(Q) \gtrsim_m  \frac{\delta\,\mu(Q)}{\mu(Q_{j})}\frac{\delta\,\mu(Q_j)}{\mu(Q_{j-1})} \, \frac{\delta\,\mu(Q_{j-1})}{\mu(Q_{j-2})}\,\ldots
\frac{\delta\,\mu(Q_1)}{\mu(R)} = \delta^{j+1}\,\frac{\mu(Q)}{\mu(R)}\gtrsim_m  \delta^{m+1} \,\frac{\mu(Q)}{\mu(R)},$$
which proves the lemma.
\end{proof}


\vv


\vv
\subsection{The $A_\infty(\mu)$ property of $\omega^p$}
To show that $\omega^p$ satisfies the $A_\infty(\mu)$ property it is enough to prove that
 there exists
some constant $\tau>0$ such that, for each $R\in\DD_\mu$ and every $F\subset R$, 
\begin{equation}\label{eqainfy}
\omega^p(F)\leq \tau\,\omega^p(R)    \quad \;\Rightarrow\; \quad \mu(F)\leq \frac34\,\mu(R).
\end{equation}

By the martingale property of harmonic measure in uniform domains with $n$-AD-regular boundary, 
the left inequality in \rf{eqainfy} implies that
\begin{equation}\label{eqainfy2}
\omega^{p_R}(F)\approx\frac{\omega^{p_R}(F)}{\omega^{p_R}(R)} \approx \frac{\omega^{p}(F)}{\omega^{p}(R)}
\leq\tau.
\end{equation}
See \cite{AiHi} for further details (or \cite[Theorem 1.3]{MTo} for a more precise reference) regarding
this martingale property of harmonic measure in uniform domains.
Now we write 
\begin{equation}\label{eqainfy3}
\mu(F)\leq \mu(B_L^m(R)) + \mu(F \setminus B_L^m(R)) \leq \frac12\,\mu(R)  + \mu(F \setminus B_L^m(R)) ,
\end{equation}
taking into account \rf{eqak941}.
Recall also that, by Lemma \ref{lemg},
$$\frac1{\mu(R)}\,\mu|_{R\setminus B_L^m(R)} = g\,\omega^{p_R}|_{R\setminus B_L^m(R)},$$
with $\|g\|_{L^\infty(\omega^{p_R}|_{R\setminus B_L^m(R)})} \leq C(\delta,m)$.
Together with \rf{eqainfy2} this implies that
$$\mu(F \setminus B_L^m(R))\leq C(\delta,m)\,\mu(R)\,\omega^{p_R}(F \setminus B_L^m(R)) \leq C(\delta,m)\,\tau\,\mu(R)
\leq \frac14\,\mu(R),$$
if $\tau$ is assumed small enough. Plugging this into \rf{eqainfy3}, we obtain \rf{eqainfy}.

\vvv

\section{From the corona decomposition for harmonic measure to uniform rectifiability}\label{sec6}

In this section we will show that the existence of a corona decomposition  such as the one described
in Proposition \ref{propo11} implies the uniform rectifiability of $\partial\Omega$. We will
prove this by showing that the Riesz transform $\RR_\mu$ is bounded in $L^2(\mu)$, and then 
applying the main theorem of \cite{NToV}.

\vv
The precise result that we will prove is the following.

\begin{propo}\label{propo12}
Let $\Omega\subset\R^{n+1}$ be a corkscrew domain with $n$-AD-regular boundary. Denote by
$\mu$ the surface measure on $\partial\Omega$.
Suppose that $\mu$ admits a corona decomposition
$\DD_\mu=\bigcup_{R\in \ttt} \tree(R)$
so that the family $\ttt$ is a Carleson family, that is,
\begin{equation}\label{eqpack57}
\sum_{R\subset S: R\in\ttt}\mu(R)\leq C\,\mu(S)\quad \mbox{for all $S\in\DD_\mu$},
\end{equation}
and for each $R\in\ttt$ there exists a corkscrew point $p_R\in\Omega$ with 
$$c^{-1}\ell(R)\leq\dist(p_R,R)\leq \dist(p_R,\partial\Omega)\leq c\,\ell(R)$$
such that, for some fixed $\lambda_0>1$,
$$
\omega^{p_R}(\lambda_0 Q)\approx \frac{\mu(Q)}{\mu(R)}\quad\mbox{ for all $Q\in\tree(R)$,}$$
with the implicit constant uniform on $Q$ and $R$.
Then $\mu$ is uniformly rectifiable. 
\end{propo}

Note that this proposition, when combined  with Proposition \ref{propo11}, completes the proof of Theorem \ref{teo1}.
Note also that from a family $\FF$  as  described in Theorem \ref{teo2} one can construct a corona
decomposition like that in Proposition 5.1, with $\lambda_0=7$. Indeed, if we let $\ttt$ be the family of children of all cubes from $\FF$, together
with $R_0=\partial\Omega$ if $\diam(\Omega)=\infty$, and then for $R\in\ttt$  we let $\tree(R)$ be the family of 
cubes $Q\in\DD_\mu(R)$ which are not contained in any cube from $\ttt\cap\DD_\mu(R)$ 
different from $R$, then it is easy to check
that the properties stated in the above proposition\footnote{We do not choose $\ttt=\FF$ because this would not guarantee the property (3) in the definition of the corona decomposition in Subsection \ref{subsecorona}.}  hold for this definition of $\ttt$.
Hence the combination of Propositions \ref{propo11} and \ref{propo12} also yields Theorem \ref{teo2}.

\vv

In the rest of this section we suppose that the assumptions of
 Proposition \ref{propo12} hold for the family $\ttt
\subset\DD_\mu$.
 Further, for simplicity we will assume that $\lambda_0=3$. Very minor modifications,
which we leave for the reader, yield the conclusion in the case $\lambda_0>1$. 

\vv

\subsection{The Riesz transform of $\omega^{p_R}$}

Given $R\in\ttt$, we denote by $\sss(R)$ the family of cubes $Q\in\tree(R)$ such that
their children do not belong to $\tree(R)$.

The connection between the Riesz transform operator and harmonic measure is provided by the following result.

\begin{lemma}\label{lemriesz1}
For $x\in R$, denote
$$\ell(x) = \left\{ \begin{array}{ll}
0 &\quad \mbox{if $x\in R\setminus\bigcup_{Q\in\sss(R)} Q$,}\\
&\\
\ell(Q)&\quad  \mbox{if $x\in Q\in\sss(R)$.}
\end{array}\right.
$$
Then,
$$\sup_{t>\ell(x)} |\RR_t \omega^{p_R}(x)|\lesssim \frac1{\mu(R)}.$$
\end{lemma}

\begin{proof}
First note that for all $t\geq 2\ell(R)$ and $x\in R$,
$$|\RR_{t}\omega^{p_R}(x)|\leq \frac{\|\omega^{p_R}\|}{\ell(R)^n} \lesssim \frac1{\mu(R)}.$$

Next we will show that for $x\in Q\in\tree(R)$ and $\ell(Q)\leq t \leq 2\ell(Q)$,  
 \begin{equation}\label{eq:trunc.Riesz.hm}
|\RR_{t}\omega^{p_R}(x)| \lesssim \frac{1}{\mu(R)}.
\end{equation}
Clearly, this suffices to prove the lemma.
Let $p_Q$ be the corkscrew point associated with the cube $Q$, as defined at the beginning of Section \ref{sec:aproximation}. Now, by standard Calder\' on-Zygmund estimates, using that 
all the ancestors of $Q$ in $\tree(R)$ satisfy 
$$\omega^{p_R}(3Q)\lesssim \frac{\mu(Q)}{\mu(R)}$$
 and that $\| \omega^{p_R}\|=1$ for the ancestors that do not belong to $\tree(R)$, 
 it is easy to prove that  
 \begin{equation*}
|\RR_{t}\omega^{p_R}(x)-\RR_{r(V_Q)}\omega^{p_R}(p_Q)| \lesssim \sup_{\lambda\geq1} \frac{\omega^{p_R}(\lambda Q)}{\ell(\lambda Q)^n}\lesssim  \frac{1}{\mu(R)}.
\end{equation*}
Notice also that by the choice of $p_Q$, 
$$ \RR_{r(V_Q)}\omega^{p_R}(p_Q) =\RR \omega^{p_R}(p_Q).$$
Therefore, to finish the proof of the lemma it is enough to show that 
 \begin{equation}\label{eq:Rieszw<mR}
 |\RR \omega^{p_R}(p_Q) |\lesssim \frac{1}{\mu(R)}.
 \end{equation}

From \rf{eqclau1}, it is clear that for all $x\in V_Q$ we have
\begin{equation}\label{eqkey6}
|\RR\omega^{p_R}(x)| \lesssim \frac1{\ell(R)^n} + |\nabla_x G(x,{p_R})|.
\end{equation}
Since $G(\cdot,{p_R})$ is harmonic in $2V_Q$ and positive in $\Omega$, for all $x\in V_Q$ we have
$$|\nabla_x G(x,{p_R})|\lesssim \frac1{r(V_Q)}\,\sup_{y\in2V_Q} |G(y,{p_R})-G(p_Q,p_R)|.$$
Then, using \eqref{e:w>g} and \rf{e:w>gpla}, along with the fact that $\mu$ has polynomial growth, we infer that
for $y\in2V_Q$
\begin{align}\label{eq:green/lQ}\frac{|G(y,{p_R})-G(p_Q,p_R)|}{r(V_Q)} &\lesssim
\frac{\omega^{p_R}(2\delta_0^{-1} Q)}
{\ell(Q)^{n}\,\omega^{p_Q}(Q)}  \\\lesssim_{\delta_0}\frac{\omega^{p_R}(2\delta_0^{-1}Q)}{\mu(2\delta_0^{-1}Q)} &\lesssim \frac{\omega^{p_R}(2R)}{\mu(2R)}\lesssim \frac{1}{\mu(R)}.
\nonumber
\end{align}
Together with 
\rf{eqkey6}, this gives \rf{eq:Rieszw<mR}.
\end{proof}

\vv


\subsection{Decomposition of $\RR_\mu$ in terms of the corona decomposition}

Given $\eta>0$ and a function $f\in L^2(\mu)$, we decompose the $\eta$-truncated Riesz transform of $f\mu$ as follows.
For every $Q\in\DD_\mu$, we set\footnote{Recall that for $Q\in\DD_{\mu,j}$, we write $\ell(Q)= 2^{-j }$ and have $\diam(Q)\leq\ell(Q)$.}
$$\RR_Q (f\mu)(x) = \chi_Q \int_{\max(\eta,\ell(Q)/2)<|x-y|\leq\ell(Q)} \frac{x-y}{|x-y|^{n+1}}\,f(y)\,d\mu(y),$$
and for every $R\in\ttt$,
$$K_R (f\mu)(x) = \sum_{Q\in\tree(R)}\RR_Q (f\mu)(x),$$
so that
\begin{equation}\label{eqdecomp}
\RR_\eta(f\mu) = \sum_{R\in\ttt}K_R (f\mu).
\end{equation}

Using the decomposition \rf{eqdecomp} we will show below that
\begin{equation}\label{eqacot}
\|\RR_\eta(f\mu)\|_{L^2(\mu|_{R_0})}\leq c\,\|f\|_{L^2(\mu)},
\end{equation}
with the constant $c$ uniform on $R_0\in\DD_\mu$ and $\eta>0$. Clearly this yields the $L^2(\mu)$-boundedness of $\RR_\mu$ and thus the uniform rectifiability
of $\mu$, by \cite{NToV}.

\vv


\subsection{The suppressed Riesz transform}

In this subsection we describe some results on singular integrals with ``suppressed kernels''. All the results
in this subsection are due to Nazarov, Treil and Volberg \cite{NTV} (although we may provide different references which may be more convenient for the reader).
We consider the ``suppressed kernel'' 
$$K_\Phi(x,y) = \frac{x-y}{\bigl(|x-y|^2 + \Phi(x)\,\Phi(y)\bigr)^{(n+1)/2}},$$
which satisfies
\begin{equation}\label{eq:supkerbound}
|K_\Phi(x,y)| \lesssim \min \left(\frac{1}{|x-y|^n}, \frac{1}{\Phi(x)^n},\frac{1}{\Phi(y)^n}\right),
\end{equation}
and 
\begin{equation}\label{eq:supkerholder}
|K_\Phi(x,y)-K_\Phi(x',y)| \lesssim \frac{|x-x'|}{|x-y|^{n+1}}, \quad \text{if} \,\, 2|x-x'| \leq |x-y|.
\end{equation}
Since $K_\Phi$ is anti-symmetric an analogous estimate holds in the $y$-variable. For a proof of the above estimates see e.g. Lemma 8.2 in \cite{Volberg}.
We define the associated Riesz transform by
$$\RR_\Phi \nu(x) :=\int K_\Phi(x,y)\,d\nu,$$
for a Radon measure $\nu$.
We also set $\RR_{\nu,\Phi} f := \RR_\Phi (f\,\nu)$, for $f\in L^1_{loc}(\nu)$.
\vv

We shall now record some auxiliary results which we will use repeatedly. If $\nu$ is a positive and finite Borel measure, we define 
\begin{equation}\label{eqrho}
\rho_\nu(x):= \sup\{r>0: \nu(B(x,r)> C_0 r^n\}
\end{equation}
and set $\rho_\nu(x)=0$, if the set on the right hand-side is empty. 
From now on we assume that $\Phi:\R^{n+1} \to [0, \infty)$ is a $1$-Lipschitz function such that $\Phi(x)\geq \rho_\nu(x)$ for any $x \in \R^{n+1}$. 

\vv
\begin{lemma}[Lemma 5.4, \cite{Tolsa-llibre}]\label{lem:5.4-llibre}
If $x \in \R^{n+1}$ and $\ve > \Phi(x)$, then
$$|\RR_{\Phi, \ve} \nu(x) - \RR_{\ve} \nu(x)| \lesssim \sup_{r>\ve} \frac{\nu(B(x,r))}{r^n}.$$
\end{lemma}
\vv


\begin{lemma}[Lemma 5.26, \cite{Tolsa-llibre}]\label{lem:5.26-llibre}
Suppose that the operator $\RR_{\nu,\Phi}$ is bounded from $L^1(\nu)$ to $L^{1,\infty}(\nu)$. If $s \in (0,1]$, for any $f \in L^2(\nu)$ it holds that
$$\RR_{\nu, \Phi, *}f (x) \lesssim M_\nu [( \RR_{\nu, \Phi}f)^s](x)^{1/s} +  (1+\|\RR_{\Phi}\|_{L^1(\nu) \to L^{1,\infty}(\nu)})  M_\nu f(x).$$
\end{lemma}

We remark here that in \cite{Tolsa-llibre} Lemma \ref{lem:5.26-llibre} is stated only for $s=1$. However, the same arguments (with minor adjustments) show this more general version as well.

\vv  
 Given $b>1$, we introduce the following suppressed  Hardy-Littlewood maximal operators: 
   \begin{align*}
 M_{\nu,b, \Phi} f(x)&:= \sup_{r  \geq\Phi(x)} \frac{1}{\nu(B(x, br))} \int_{B(x, br)} |f(y)| \,d\nu(y),\\
 M^r_{\nu,\Phi} f(x)&:= \sup_{r  \geq\Phi(x)} \frac{1}{r^n} \int_{B(x, b r)} |f(y)| \,d\nu(y).
 \end{align*}

We will need the following variant of Lemma \ref{lem:5.26-llibre}. The proof follows by
inspection from the proof of Lemma \ref{lem:5.26-llibre} and we leave the details to the reader.

\begin{lemma}\label{lem:var5.26}
Suppose that the operator $\RR_{\nu,\Phi}$ is bounded from $L^1(\nu)$ to $L^{1,\infty}(\nu)$.  If $s \in (0,1]$, then for any $f \in L^2(\nu)$, 
$$\RR_{\Phi, *}(f \,\nu) (x) \lesssim M_{\nu, b, \Phi} ( \RR_{\Phi}(f \,\nu)^s)(x)^{1/s} +  (1+\|\RR_{\Phi}\|_{L^1(\nu) \to L^{1,\infty}(\nu)})  M^r_{\nu,\Phi} f(x) + M_{\nu, b, \Phi} f(x),$$
with the implicit constant depending on $s$ and $b$.
\end{lemma}

\vv
\begin{lemma}[Lemma 5.27, \cite{Tolsa-llibre}]\label{lem:5.27-llibre}
If the operator $\RR_{\nu,\Phi}$ is bounded in $L^2(\nu)$, then it is also bounded from $L^1(\nu)$ to $L^{1,\infty}(\nu)$. Moreover, 
$$\|\RR_{\nu,\Phi}\|_{L^1(\nu) \to L^{1,\infty}(\nu)} \lesssim 1+ \|\RR_{\nu,\Phi}\|_{L^2(\nu) \to L^2(\nu)}. $$
\end{lemma}
\vv

\begin{theorem}\label{lem:12.1-Vol}
If there exists a constant $C_1>0$ such that $\RR_{\Phi, *} \nu(x)\leq C_1$ for $\nu$-a.e.\ $x$, then  $\RR_{\nu,\Phi}$ is bounded in $L^2(\nu)$.
\end{theorem}
 
\begin{proof}
This follows from an application of the $Tb$ theorem for suppressed operators of Nazarov, Treil and Volberg (see Theorem 12.1 in \cite{Volberg}). Indeed, if our test function is $b=1$, then it is always accretive and thus, the set $T_{12}$ in Theorem 12.1 is just the empty set. Therefore, since 
$$|K_\Phi(x,y)| \lesssim \min \left(\frac{1}{|x-y|^n}, \frac{1}{\Phi(x)^n},\frac{1}{\Phi(y)^n}\right),$$
$\Phi(x)\geq \rho_\nu(x)$ and, by assumption, $\RR_{\nu, \Phi, *} \leq C_1$, we can apply Theorem 12.1 in \cite{Volberg} to obtain that $\RR_{\Phi}$ is bounded in $L^2(\nu)$.
\end{proof}
\vvv

\subsection{The $L^p(\mu)$-boundedness of $K_R(\cdot\,\mu)$}

Our next objective consists in proving the following:

\begin{lemma}\label{lemkr}
Let $1<p<\infty$ and $f\in L^p(\mu)$. Then for every $R\in\ttt$ we have
$$\|K_R(f\mu)\|_{L^p(\mu)}\lesssim\|f\|_{L^p(\mu)},$$
with the implicit constant uniform on $R$.
\end{lemma}
\vv

We need to consider the following auxiliary function:
$$\Phi(x) = \inf_{Q\in\tree(R)} \bigl(\dist(x,Q) + \ell(Q)\bigr),\quad  x\in\R^{n+1}.$$
Notice that $\Phi$ is $1$-Lipschitz.
Next we define a ``regularized'' family  $\reg$ of cubes.
 For each $x\in
\supp\mu$ such that $\Phi(x)\neq0$, let $Q_x$ be a dyadic cube from
$\DD_\mu$ containing $x$ such that $$\frac{\Phi(x)}{2} <
\ell(Q_x) \leq {\Phi(x)}.$$ 
If $\Phi(x)=0$, we set $Q_x=\{x\}$.
Then, $\reg(R)$ is a maximal
(and thus disjoint) subfamily of $\{Q_x\}_{x\in 3R,\Phi(x)>0}$. 
Note that not all the cubes from $\reg(R)$ are contained
in $R$.
\vv

\begin{lemma} \label{lemregul} The family $\reg(R)$ satisfies:
\begin{itemize}
\item[(a)] $\bigcup_{Q\in\reg(R)\cap\DD_\mu(R)} Q\subset
\bigcup_{Q\in\sss(R)}Q$.
\vv
\item[(b)] If $P,Q\in \reg(R)$ and $2P\cap 2Q\neq
\varnothing$, then $\ell(Q)/2 \leq \ell(P) \leq 2 \ell(Q)$.
\vv
\item[(c)] If $Q\in
\reg(R)$, $Q\cap 3R\neq\varnothing$, and $x\in Q$, $r\geq \ell(Q)$, then 
$$\omega^{p_R}(B(x,r)) \lesssim r^n/\mu(R).$$
\vv
\item[(d)] For each $Q\in \reg(R)$ with $Q\cap 3R\neq\varnothing$, there exists some cube $\wt{Q}\in\tree(R)$
 such that 
 $$\ell(\wt Q)\approx\ell(Q)\quad \mbox{ and }\quad \dist(Q,\wt Q)\lesssim \ell(Q).$$
\end{itemize}
\end{lemma}

The proof of this lemma follows by standard arguments (see e.g. Theorem 8.2 in \cite{Tolsa-Annals} or Lemma 3 in \cite{G-S}). 

We remark, that abusing notation, we may also think of the points $x\in\R^d$ such that $\Phi(x)=0$ as cubes
with side length $0$. Then, if we enlarge the family $\reg(R)$ by adding these cubes consisting of
a single point, the resulting family, call it $\wh\reg(R)$, also satisfies the properties in the preceding lemma.

By the properties in Lemma \ref{lemregul} and an easy application of the Lebesgue differentiation
theorem, one can see that, module a set of zero $\mu$ and $\omega^{p_R}$ measure. 
\begin{equation}\label{eqzz0}
R\setminus \bigcup_{Q\in\sss(R)} Q = \supp\mu\setminus \bigcup_{Q\in\reg(R)} Q
\end{equation}

\vv
From now on, in this subsection we denote by $\wt\mu$ and $\sigma$ the measures
$$\wt\mu = \mu|_{\bigcup_{Q\in\reg(R)} 3R\cap Q},\qquad
\sigma = \mu(R)\,\omega^{p_R}|_{3R}.$$

\vv
\begin{lemma}\label{lemgrosimga}
For all $x\in3R$, we have
\begin{equation}\label{eqakt88}
\sigma(B(x,r))\leq C\,r^n\quad \mbox{ for all $r\geq \Phi(x)$.}
\end{equation}
\end{lemma}

\begin{proof}
By the definition of $\Phi(x)$, there exists some $Q\in\tree(R)$ such that
$$
\Phi(x) \leq 2\bigl(\dist(x,Q) + \ell(Q)\bigr).
$$
Hence there exists some cube
$S\in\DD_\mu$ such that $S\supset Q$ and $3S\supset B(x,r)$ with $\ell(S)\approx r$.
In particular, either $S\in\tree(R)$ or $S\supset R$ and since
$$\omega^{p_R}(3S)\lesssim \frac{\mu(S)}{\mu(R)},$$
we have
$$\sigma(B(x,r))\leq\sigma(3S) = \mu(R)\,\omega^{p_R}(3S\cap 3R) \lesssim  \mu(S)
\approx \ell(S)^n\approx r^n.$$
\end{proof}

\vv
From the last lemma it follows that $$\Phi(x)\geq \rho_\sigma(x),$$
if we choose the constant $C_0$ in the definition \rf{eqrho} of $\rho_\sigma$ to be the constant 
$C$ on the right hand side of \rf{eqakt88}.
\vv

\begin{lemma}\label{lemfac999}
In $3R\setminus\bigcup_{Q\in\reg(R)} Q$, we have $d\sigma(x) = h(x)\,d\mu(x)$, with $h(x)\approx1$. 
\end{lemma}

\begin{proof}
By the Lebesgue differentiation theorem, it follows that
$\omega^{p_R}\ll\mu$ in $R\setminus\bigcup_{Q\in\sss(R)} Q$, with 
$d\omega^{p_R}=h(x)\frac1{\mu(R)}\mu$ and $h(x)\approx1$ on this set, which yields the lemma.
\end{proof}

\vv

\begin{lemma}\label{lem:reg-growth}
If $Q\in \reg(R)$ and $x\in Q$, $\ell(Q)\leq r\leq 100\,\diam(3R)$, then there exists some constant $b\geq 1$ such that 
\begin{equation}\label{eq:reg-growth}
\mu(B(x,r)) \lesssim \omega^{p_R}(B(x, br) \cap 3R)\, \mu(R).
\end{equation}
The bounds on the constant $b$ only depend on the parameters of the construction of $\DD_\mu$.
\end{lemma}

\begin{proof}
Given $x \in Q$, there exist $Q'\in \tree(R)$ and $x_{Q'} \in Q'$ such that 
$$\ell(Q) \leq |x-x_{Q'}| + \ell(Q')\leq 2\ell(Q).$$ 
From the above inequalities it is clear that $|x-x_{Q'}| \leq 2 r$ and $\ell(Q')/2 \leq r$. Therefore there are two cases:

\noindent {\bf Case (i):} There exists $Q'' \in \tree(R)$ such that $B(x, r)\cap\supp\mu \subset 3 Q''$ and $\ell(Q'') \approx r$. Trivially, one can find $b \approx 1$ such that  $Q'' \subset B(x, b r) \cap 3R$. Therefore, 
\begin{align*}
\mu(B(x,r)) &\leq  \mu(3 Q'')\lesssim \mu(Q'')\lesssim \hm^{p_R} (3Q'')\, \mu(R) \leq  \hm^{p_R}(B(x, b r) \cap 3R)\,\mu(R),
\end{align*}
where in the penultimate inequality we used that $\omega^{p_R}(3Q'')\approx \frac{\mu(Q'')}{\mu(R)}$.

\noindent {\bf Case (ii):} $3R \subset B(x,r)$. Since $r\leq100\,\diam(3R)$, for some $C>1$,
\begin{align*}
\mu(B(x,r)) &\leq \mu(CR) \lesssim  \hm^{p_R} (R)\, \mu(R) \leq  \hm^{p_R}(B(x, r) \cap 3R)\,\mu(R),
\end{align*}
which concludes our lemma.
 \end{proof}

 \vv
\begin{lemma}\label{lem:maximal}
The operators $M_{\sigma,b, \Phi}$ and $M^r_{\sigma,\Phi}$ are bounded from $L^p(\sigma)$ to $L^p(\wt\mu)$, $1<p\leq\infty$,
and from $L^1(\sigma)$ to $L^{1,\infty}(\wt\mu)$, with norms depending $p$.
\end{lemma}

 \begin{proof}
The boundedness of $M^r_{\sigma,\Phi}$ is well known; see for instance in the proof of Lemma 7.6 in \cite{Tolsa-pubmat}. Concerning $M_{\sigma, b,\Phi}$, it is enough to show that this is bounded from $L^1(\sigma)$ to $L^{1,\infty}(\wt\mu)$, since this is trivially bounded from  $L^\infty(\sigma)$ to $L^\infty(\wt\mu)$. To this end, let $f \in L^1(\sigma)$ and for fixed $\lambda>0$ set 
$$ \Omega_\lambda:= \Bigl\{ x \in \textstyle{\bigcup_{Q\in\reg(R)}} 3R\cap Q: M_{\sigma,b, \Phi}f(x)>\lambda\Bigr\}.$$
By definition, for each $x \in \Omega_\lambda$, there exists a ball $B_x$ centered at $x$, with radius $r(B_x)\geq\Phi(x)$ such that
$$ \frac{1}{\lambda} \int_{bB_x} |f(y)| \,d\sigma(y)> \sigma(bB_x)=\hm^{p_R}(bB_x \cap 3R)\,\mu(R).$$
Further, we may assume that $r(B_x)\leq\diam(3R)$ because $\supp\sigma\cup \supp\wt\mu\subset \overline{3R}$.
By Vitali's $5r$-covering lemma, we can find a countable family of balls $\{B_i\}$ such that $b B_i \cap b B_j = \varnothing$ for $i \neq j$, and 
$$\bigcup_{x \in \Omega_\lambda} B_x \subset \bigcup_i 5b B_i.$$
Since each ball $B_i$ is centered at some point $x\in Q\in \reg(R)$, by Lemma \ref{lem:reg-growth} we have that
$$\mu(B_i)\leq \hm^{p_R}(b B_i \cap 3R)\,\mu(R).$$
Then we deduce
\begin{align*}
\wt \mu (\Omega_\lambda) &\leq \sum_i \wt \mu(5b B_i) \lesssim_b \sum_i \mu( B_i) \lesssim \sum_i  \hm^{p_R}(b B_i \cap 3R)\,\mu(R) \\
&\leq \sum_i  \frac{1}{\lambda} \int_{b B_{i}} |f(y)| \,d\sigma(y) \leq \frac{\|f\|_{L^1(\sigma)}}{\lambda},
\end{align*}
which finishes our proof.
\end{proof}
\vv

\begin{lemma}\label{lemsigma1}
The operator $\RR_{\sigma,\Phi}$ is bounded in $L^p(\sigma)$, $1<p<\infty$, and from $L^1(\sigma)$ to $L^{1,\infty}(\sigma)$, with norm depending on $p$.
\end{lemma}

\begin{proof}
We first prove that $\RR_{\sigma,\Phi}$ is bounded in $L^p(\sigma)$  for $1<p<\infty$. 
Taking into account Lemma \ref{lemgrosimga}, by Theorem \ref{lem:12.1-Vol}, it is enough to show that
\begin{equation}\label{eq:sigma*9}
\RR_{\Phi, *}\sigma (x)\lesssim 1 \quad \mbox{ for all $x \in 3R$.}
\end{equation}

For any $x \in  3R$, we write
\begin{equation}\label{eq:Riesz-split}
\RR_{\Phi, *}[\chi_{3R}\,\hm^{p_R}] (x) \leq \RR_{\Phi, *} \hm^{p_R}(x)  + \RR_{\Phi, *} [\chi_{(3R)^c}\,\hm^{p_R}](x).
\end{equation}
Let us estimate the first term on the right hand side.
Suppose first that $x\in Q\in \reg(R)$, with $Q\cap 3R\neq \varnothing$. By Lemma \ref{lemregul} (d), there
exists some cube $\wt{Q}\in\tree(R)$ such that 
 $$\ell(\wt Q)\approx\ell(Q)\quad \mbox{ and }\quad \dist(Q,\wt Q)\lesssim \ell(Q).$$
By Lemma \ref{lemriesz1}, it holds that for any $\wt x \in \wt Q$, 
\begin{equation}\label{eqriesz857}
\sup_{t>\ell(\wt Q)} |\RR_t \omega^{p_R}(\wt x)|\lesssim \frac1{\mu(R)}.
\end{equation}
By standard estimates, it follows that
$$\sup_{t>\ell(Q)} |\RR_t \omega^{p_R}(x)|\lesssim
\sup_{t>\ell(\wt Q)} |\RR_t \omega^{p_R}(\wt x)| + \sup_{t>\ell(\wt Q)} \frac{\omega^{p_R}(B(\wt x,t))}{t^n}.
$$
From the properties of the cubes from $\tree(R)$, it follows easily that
$$\sup_{t>\ell(\wt Q)} \frac{\omega^{p_R}(B(\wt x,t))}{t^n} \lesssim \frac1{\mu(R)}.$$
Together with \rf{eqriesz857}, this gives
\begin{equation}\label{eq:RtlQ-reg}
\sup_{t>\ell(Q)} |\RR_t \omega^{p_R}(x)|\lesssim \frac1{\mu(R)}.
\end{equation}
This, in turn, implies that if $x \in  \bigcup_{Q \in \reg(R)} Q$,
\begin{equation}\label{eqriesz858}
\RR_{\Phi, *} \hm^{p_R}(x) \lesssim \frac1{\mu(R)}.
\end{equation}
Indeed, notice that if $x \in Q \in \reg(R)$ and $\ve \leq 2 \Phi(x)$, then by standard estimates, we have that
$$ |\RR_{\Phi, \ve} \hm^{p_R}(x)  - \RR_{\Phi, 2\Phi(x)} \hm^{p_R}(x)|  \lesssim \sup_{r>2\Phi(x)}\frac{\hm^{p_R}(B(x, r))}{r^n} \leq \sup_{r>\ell(Q)}\frac{\hm^{p_R}(B(x, r))}{r^n} \lesssim \frac1{\mu(R)},$$
where in the penultimate inequality we used that $\Phi(x) \geq \ell(Q)/2$ and in the last one we used Lemma \ref{lemregul} (c). Moreover, by Lemma \ref{lem:5.4-llibre} and similar considerations,
$$ |\RR_{2\Phi(x)} \hm^{p_R}(x)  - \RR_{\Phi, 2\Phi(x)} \hm^{p_R}(x)|  \lesssim \sup_{r>2\Phi(x)}\frac{\hm^{p_R}(B(x, r))}{r^n} \lesssim \frac1{\mu(R)}.$$
The latter two estimates combined with \eqref{eq:RtlQ-reg} imply that for $\ve \leq 2 \Phi(x)$,
$$ |\RR_{\Phi, \ve} \hm^{p_R}(x)|  \lesssim \frac1{\mu(R)} + |\RR_{2\Phi(x)} \hm^{p_R}(x)|\lesssim  \frac1{\mu(R)}.$$
On the other hand, in view of \eqref{eq:RtlQ-reg}, it is clear that
$$\sup_{\ve>2 \Phi(x)} |\RR_{\Phi, \ve} \hm^{p_R}(x)|  \leq \sup_{\ve>\ell(Q)} |\RR_{\Phi, \ve} \hm^{p_R}(x)|\lesssim \frac1{\mu(R)},$$
which concludes \eqref{eqriesz858}.

In the case  $x \in 3R\setminus\bigcup_{Q\in\reg(R)} Q$, we have $\Phi(x)=0$ and a direct application of Lemma \ref{lemriesz1} shows that \rf{eqriesz858} also holds.

Next we estimate the last term in \rf{eq:Riesz-split}. To this end, note that $\Phi(y)\gtrsim \ell(R)$
for all $y\in (3R)^c$. Hence, for $x\in 3R$ and $y\in (3R)^c$, we have
$$|K_\Phi(x,y)|\lesssim \frac1{\Phi(y)^n} \lesssim \frac1{\ell(R)^n}.$$
So we get
$$\RR_{\Phi, *} [\chi_{(3R)^c}\,\hm^{p_R}](x)\lesssim \frac{\omega^{p_R}((3R)^c)}{\ell(R)^n}\leq
\frac{1}{\ell(R)^n}\lesssim \frac1{\mu(R)}.$$
In combination with \rf{eqriesz858}, this gives
$$\RR_{\Phi, *}[\chi_{3R}\,\hm^{p_R}] (x)\lesssim \frac1{\mu(R)}
\quad \mbox{ for all $x \in 3R$.}$$
This finishes the proof of \rf{eq:sigma*9} and of the $L^2(\sigma)$ boundedness of $\RR_{\sigma,\Phi}$. Together with Lemma \ref{lem:5.27-llibre}, this implies that $\RR_{\sigma,\Phi}$ is bounded from $L^1(\sigma)$ to $L^{1,\infty}(\sigma)$, and thus in $L^p(\sigma)$ for $1<p<\infty$. Our lemma is now concluded.
\end{proof}
\vv


\begin{lemma}\label{lemsigma2}
The operator $\RR_{\sigma,\Phi}$ is bounded from $L^p(\sigma)$ to $L^p(\wt\mu)$, $1<p<\infty$,
and from $L^1(\sigma)$ to $L^{1,\infty}(\wt\mu)$, with the norms depending on  $p$.
\end{lemma}
\begin{proof}

Note that Lemma \ref{lem:var5.26} holds for $\nu=\sigma$. Then that $\RR_{\sigma,\Phi}$ is bounded from $L^p(\sigma)$ to $L^p(\wt\mu)$  follows by a direct application of Lemma \ref{lem:maximal} and Lemma \ref{lemsigma1}. The same lemmas also imply that $\RR_{\sigma,\Phi}$ is bounded from $L^1(\sigma)$ to $L^{1,\infty}(\wt\mu)$ in a non-trivial way. Although the arguments are standard, we will give the proof for the sake of clarity. 

In view of Cotlar's inequality in Lemma \ref{lem:var5.26} and Lemma \ref{lem:maximal},  it suffices to prove that
for $f \in L^1(\sigma)$ and  $\lambda>0$ it holds that
$$ \wt \mu ( \{ (M_{\sigma, b, \Phi}(\RR_{\Phi, \sigma}f^s))^{1/s} > \lambda\} ) \lesssim \lambda^{-1}  \|f\|_{L^{1}(\sigma)}.$$
 Define now
$$g:=|\RR_{\sigma,\Phi} f|^s, \quad g_1:= g \,\chi_{\{|g|< \lambda^s/2\}}\quad \textup{and}\quad g_2:=g-g_1.$$
Set also $E_\lambda:=\{|\RR_{\sigma,\Phi} f|>\frac{\lambda}{2^{1/s}}\}$. Since
$$ \wt \mu ( \{ (M_{\sigma, b, \Phi}g)^{1/s} > \lambda\} ) \leq  \wt \mu ( \{ M_{\sigma, b, \Phi} g_1 > \lambda^s/2\} )+ \wt \mu ( \{ M_{\sigma, b, \Phi} g_2 > \lambda^s/2\} )$$
and $M_{\sigma, b, \Phi} g_1 \leq \lambda^s/2$, it is enough to prove that 
$$\wt \mu ( \{ M_{\sigma, b, \Phi} g_2 > \lambda^s/2\} ) \lesssim \|f\|_{L^1(\sigma)}.$$
To this end, in light of Lemma \ref{lem:maximal}, Kolmogorov's inequality (see e.g. Lemma 2.19 in \cite{Tolsa-llibre}) and  Lemma \ref{lemsigma1} (i.e. $\RR_{\sigma,\Phi}$ is bounded from $L^1(\sigma)$ to $L^{1,\infty}(\sigma)$), we get that
\begin{align*}
\wt \mu ( \{ M_{\sigma, b, \Phi} g_2 > \lambda^s/2\} ) &\lesssim \lambda^{-s} \int|g_2| \,d\sigma=\lambda^{-s} \int_{E_\lambda}|\RR_{\sigma,\Phi} f|^s \,d\sigma\\
&\lesssim \lambda^{-s}\, \sigma(E_\lambda)^{1-s}\, \|\RR_{\sigma,\Phi} f\|_{L^{1,\infty}(\sigma)}^s\\
&\lesssim \lambda^{-s}\, \lambda^{s-1}\|f\|_{L^{1}(\sigma)}^{1-s}\,\|f\|_{L^{1}(\sigma)}^s\\
&= \lambda^{-1}\|f\|_{L^{1}(\sigma)}.
\end{align*}
This finishes the proof of the lemma.
\end{proof}
\vv
\begin{rem}\label{rem:duality}
Since $K_{\Phi}$ is antisymmetric, by duality, $\RR_{\wt\mu,\Phi}: L^p(\wt \mu) \to L^p(\sigma)$ is bounded, for $1<p<\infty$.
\end{rem}
\vv

Next we intend to show that $\RR_{\wt\mu,\Phi}$ is bounded in $L^p(\wt\mu)$.

\vv
\begin{lemma}\label{lemsigma3}
The operator $\RR_{\wt\mu,\Phi}$ is bounded in $L^p(\wt\mu)$, for $1<p<\infty$, with its norm depending on $p$.
\end{lemma}

\begin{proof}
It is enough to prove that $\RR_{\wt\mu,\Phi}$ is bounded from $L^1(\wt\mu)$ to $L^{1,\infty}(\wt\mu)$ because as shown in \cite[Proposition 7.8]{Tolsa-pubmat}, for example, this implies the $L^p(\mu)$-boundedness for $1<p<\infty$.
To this end,  let $f\in L^1(\wt\mu)$ and
for each $Q\in\reg(R)$ such that $Q\cap 3R\neq\varnothing$, consider the function $\vphi_Q$ defined by
$$\vphi_Q = \chi_{aQ}\, \frac1{\sigma(aQ)}\int_Q f\,d\wt\mu,$$
where $a>1$ will be fixed in a moment. In this way, we have
$$ \|\vphi_Q\|_{L^1(\sigma)} = \int_{a Q} \vphi_Q\,d\sigma = \int_Q f\,d\wt\mu.$$
Further, if $a$ is chosen big enough (i.e., $a\gtrsim b$), by Lemma \ref{lem:reg-growth} we have
$\sigma(aQ)\gtrsim \mu(Q)$, and so
\begin{equation}\label{eqsum949}
\|\vphi_Q\|_{L^2(\sigma)}\leq \frac1{\sigma(aQ)^{1/2}}\int_Q f\,d\wt\mu\lesssim \frac1{\mu(Q)^{1/2}}\int_Q f\,d\wt\mu.
\end{equation}

Now we we write
$$f\,\wt\mu = \sum_{Q\in\reg(R)} \bigl(f\,\wt\mu|_Q - \vphi_Q\,\sigma\bigr) +  \sum_{Q\in\reg(R)} \vphi_Q\,\sigma =: \nu + \eta,$$
which implies that
\begin{equation}\label{eqsum94}
\wt\mu\bigl(\bigl\{x:|\RR_{\wt\mu,\Phi}f(x)|>\lambda\bigr\}\bigr)\leq
\wt\mu\bigl(\bigl\{x:|\RR_{\Phi}\nu(x)|>\lambda/2\bigr\}\bigr) +
\wt\mu\bigl(\bigl\{x:|\RR_{\Phi}\eta(x)|>\lambda/2\bigr\}\bigr).
\end{equation}
To deal with the last term above, we use the boundedness of $\RR_{\sigma,\Phi}$ from $L^1(\sigma)$ to $L^{1,\infty}(\wt\mu)$, proved in 
Lemma \ref{lemsigma2}:
$$\wt\mu\bigl(\bigl\{x:|\RR_{\Phi}\eta(x)|>\lambda/2\bigr\}\bigr)\leq C\,\frac{\bigl\|\sum_{Q\in\reg(R)}\vphi_Q\bigr\|_{L^1(\sigma)}}\lambda.$$
Observe now that
$$\sum_{Q\in\reg(R)}\bigl\|\vphi_Q\bigr\|_{L^1(\sigma)}\leq \sum_{Q\in\reg(R)}
\int_Q |f|\,d\wt\mu =\|f\|_{L^1(\wt\mu)}.$$
Hence,
$$\wt\mu\bigl(\bigl\{x:|\RR_{\Phi}\eta(x)|>\lambda/2\bigr\}\bigr)\leq C\,\frac{\|f\|_{L^1(\wt\mu)}}\lambda.$$

To estimate the first term on the right hand side of \rf{eqsum94}, we set
$$\wt\mu\bigl(\bigl\{x:|\RR_{\Phi}\nu(x)|>\lambda/2\bigr\}\bigr)\leq \frac1\lambda\int |\RR_{\Phi}\nu|\,d\wt\mu 
\leq \frac1\lambda
\sum_{Q\in\reg(R)}\int |\RR_{\Phi}\nu_Q|\,d\wt\mu,$$
where we wrote
$$\nu_Q= f\,\wt\mu|_Q - \vphi_Q\,\sigma.$$
Now we split
\begin{align}\label{eqsum95}
\int |\RR_{\Phi}\nu_Q|\,d\wt\mu & = \int_{2aQ} |\RR_{\Phi}\nu_Q|\,d\wt\mu +
\int_{(2aQ)^c} |\RR_{\Phi}\nu_Q|\,d\wt\mu\\
& \leq \int_{2aQ} |\RR_{\Phi} (f\,\wt\mu|_Q)|\,d\wt\mu +
\int_{2aQ} |\RR_{\Phi}(\vphi_Q\,\sigma)|\,d\wt\mu + \int_{(2aQ)^c} |\RR_{\Phi}\nu_Q|\,d\wt\mu.\nonumber 
\end{align}
For the first summand on the right hand side, using that $\Phi(x)\approx \ell(Q)$ for all $x\in Q$, we get
\begin{align*}
\int_{2aQ} |\RR_{\Phi} (f\,\wt\mu|_Q)|\,d\wt\mu & \leq \mu(2aQ)\,\|\RR_{\Phi}(f\,\wt\mu|_Q)\|_{L^\infty(\wt\mu)}\\
& \lesssim\mu(2aQ)\,\frac{\|\chi_Qf\|_{L^1(\wt\mu)}
}{\ell(Q)^n}\lesssim
\|\chi_Qf\|_{L^1(\wt\mu)}.
\end{align*}
For the second summand on the right hand side of \rf{eqsum95} we use that $\RR_{\sigma,\Phi}$ is bounded
from $L^2(\sigma)$ to $L^2(\wt\mu)$, by Lemma \ref{lemsigma2}:
$$\int_{2aQ} |\RR_{\Phi}(\vphi_Q\,\sigma)|\,d\wt\mu \leq 
\wt\mu(2aQ)^{1/2}\,\|\RR_{\Phi}(\vphi_Q\,\sigma)\|_{L^2(\wt\mu)} \lesssim  
\mu(Q)^{1/2}\,\|\vphi_Q\|_{L^2(\sigma)}.$$
Using the estimate \rf{eqsum949} for $\|\vphi_Q\|_{L^2(\sigma)}$, we derive
$$\int_{2aQ} |\RR_{\Phi}(\vphi_Q\,\sigma)|\,d\wt\mu \lesssim \|\chi_Qf\|_{L^1(\wt\mu)}.
$$

To bound the last integral in \rf{eqsum95}, we take into account that $\int d\nu_Q=0$, and so for all
$x\not \in2aQ$, 
$$|\RR_{\Phi}\nu_Q(x)| \leq \int| K_\Phi(x,y)- K_\Phi(x,z_Q)|\,d|\nu_Q|(y),$$
where $z_Q$ denotes the center of $Q$. Since $| K_\Phi(x,y)- K_\Phi(x,z_Q)|\lesssim \ell(Q)/|x-z_Q|^{n+1}$,
we derive
$$|\RR_{\Phi}\nu_Q(x)| \lesssim \frac{\|\nu_Q\|}{|x-z_Q|^{n+1}},$$
and thus
$$\int_{(2aQ)^c} |\RR_{\Phi}\nu_Q|\,d\wt\mu\lesssim \|\nu_Q\|\int_{(2aQ)^c}\frac{1}{|x-z_Q|^{n+1}}\,d\wt\mu
\lesssim \|\nu_Q\|\lesssim \|\chi_Qf\|_{L^1(\wt\mu)},
$$
by standard estimates, using the polynomial growth of $\wt\mu$.

Gathering the estimates for the three terms on right hand side of \rf{eqsum95} we obtain 
$$\int |\RR_{\Phi}\nu_Q(x)|\,d\wt\mu\lesssim  \|\chi_Qf\|_{L^1(\wt\mu)},
$$
which gives
$$\wt\mu\bigl(\bigl\{x:|\RR_{\Phi}\nu(x)|>\lambda/2\bigr\}\bigr)\lesssim \frac{\sum_{Q\in\reg}\|\chi_Qf\|_{L^1(\wt\mu)}}\lambda= \frac{\|f\|_{L^1(\wt\mu)}}\lambda,$$
and completes the proof of the boundedness of $\RR_{\wt\mu,\Phi}$ from $L^1(\wt\mu)$ to $L^{1,\infty}(\wt\mu)$.
\end{proof}
\vv

\begin{lemma}\label{lemsigma4}
The operator $\RR_{\mu|_{3R},\Phi}$ is bounded in $L^p(\mu|_{3R})$, for $1<p<\infty$, with its norm depending on $p$.
\end{lemma}

\begin{proof}
We first notice that
$$\RR_{\mu|_{3R},\Phi} f = \RR_{\wt \mu,\Phi} f +\RR_{\mu|_{3R\setminus\bigcup_{Q\in\reg} Q},\Phi} f $$
Recall now that, by Lemma \ref{lemfac999}, $d\sigma(x) = h(x)\,d\mu(x)$ on $3R\setminus\bigcup_{Q\in\reg} Q$, with $h(x)\approx1$. Therefore,
\begin{align*}
\int_{3R} |\RR_{\wt \mu,\Phi}f|^p\, d\mu &= \int |\RR_{\wt \mu,\Phi}f|^p\, d\wt \mu + \int_{3R\setminus\bigcup_{Q\in\reg} Q}|\RR_{\wt \mu,\Phi}f|^p\, d\mu\\
& \lesssim \int |f|^p \,d\wt\mu + \int_{3R\setminus\bigcup_{Q\in\reg} Q}|\RR_{\wt \mu,\Phi}f|^p\, d\sigma\\
&  \lesssim \int |f|^p \,d\wt\mu,
\end{align*}
where in the last inequality we used the  boundedness of $\RR_{\wt \mu,\Phi}$ from $L^p(\wt\mu)$ to $L^p(\sigma)$, by Remark \ref{rem:duality}.

Without loss of generality, by Lemma \ref{lemsigma3}, we may assume now that $f$ is supported in $3R\setminus\bigcup_{Q\in\reg} Q$. Thus, in view of Lemmas \ref{lemsigma1} and \ref{lemsigma2}, we have that
\begin{align*}
\int_{3R} |\RR_{\mu|_{3R\setminus\bigcup_{Q\in\reg} Q},\Phi}f|^p\, d\mu &= \int |\RR_{\sigma,\Phi} (f h^{-1})|^p\, d\wt \mu + \int_{3R\setminus\bigcup_{Q\in\reg} Q}|\RR_{\sigma,\Phi}(f h^{-1})|^p\, d\mu\\
& \lesssim \int |f h^{-1}|^p \,d\sigma+ \int |\RR_{\sigma,\Phi}(f h^{-1})|^p\, d\sigma\\
&\lesssim \int |f h^{-1}|^p \,d\sigma \approx \int |f|^p \,d\mu,
\end{align*}
where we repeatedly used that $d\sigma(x) = h(x)\,d\mu(x)$ on $3R\setminus\bigcup_{Q\in\reg} Q$, with $h(x)\approx1$.
\end{proof}
\vv

\begin{lemma}\label{lem:R*bddness}
The operator $\RR_{\Phi,\mu|_{3R},*}$ is bounded in $L^p(\mu|_{3R})$, for $1<p<\infty$, with its norm depending on $p$. 
\end{lemma}
\begin{proof}
This follows immediately from Lemma \ref{lemsigma4} in conjunction with Lemma \ref{lem:var5.26} for $s=1$ and Lemma \ref{lem:5.27-llibre}.
\end{proof}
\vv

\begin{lemma}\label{lem:KboundR*}
Let $R \in \ttt$ and $x \in R$. Then 
\begin{equation}\label{eq:KboundR*}
|K_R (f\mu)(x)| \lesssim \RR_{\Phi, *}(\chi_{3R}\,f\mu)(x) + M_\mu f(x).
\end{equation}
\end{lemma}
\begin{proof}
Let us recall that $K_R(f\mu)(x)= \sum_{Q \in \tree(R)} \RR_Q(f\mu)(x)$. Notice that if $x \in R$ then either $x \in Q$ for some $Q \in \sss(R)$ or $x \in R \setminus \bigcup_{Q \in \sss(R)}Q$. 

\noindent{\bf  Case (i)}: Let $x \in Q$ for some $Q \in \sss(R)$. Then 
\begin{align*}
|K_R(f\mu)(x)| &=\left|\int_{\max(\ell(Q)/2,\eta)< |x-y| \leq\ell(R)} \chi_{3R}(y) f(y)\,K(x-y) \,d\mu(y)\right|\\
&\leq|\RR_{\max(\ell(Q)/2,\eta)}(\chi_{3R}f\mu)(x)| + |\RR_{\ell(R)}(\chi_{3R}f\mu)(x)|\\
&\lesssim |\RR_{\max(\ell(Q)/2,\eta)}(\chi_{3R}f\mu)(x)| + M_\mu f(x).
\end{align*}
By the definition of $\Phi$, $\Phi(x) \leq \ell(Q)$, and so by Lemma \ref{lem:5.4-llibre}, 
\begin{align*}
|K_R(f\mu)(x)| &\lesssim  |\RR_{\Phi, \max(\ell(Q)/2,\eta)}(\chi_{3R}f\mu)(x)| + M_\mu f(x)\\
& \lesssim |\RR_{\Phi, *}(\chi_{3R}\mu)(x)| + M_\mu f(x).
\end{align*}

\noindent{\bf  Case (ii)}: Let $x \in R \setminus \bigcup_{Q \in \sss(R)}Q$. Then every cube $P \subset R$ such that $x \in P$ is in $\tree(R)$. Thus, it is clear that
\begin{align*}
|K_R(f\mu)(x)| &=\left|\int_{\eta<|x-y| \leq\ell(R)} \chi_{3R}(y) f(y)K(x-y) \,d\mu(y)\right|.  
\end{align*}
Arguing as in the previous case and using that $\Phi(x)=0<\ell(P)$, for every $P \in \DD(R)$ such that $x \in P$ (since it is in $\tree(R)$), we can prove \eqref{eq:KboundR*}. This concludes our lemma.
\end{proof}
\vvv

\begin{proof}[\bf Proof of Lemma \ref{lemkr}]
This is an immediate consequence 
of Lemma \ref{lem:KboundR*},  since both $\RR_{\Phi,\mu|_{3R},*}$ and $M_\mu$ are bounded in $L^2(\mu)$. 
\end{proof}
\vv

\subsection{The boundedness of $\RR_\mu$ in $L^2(\mu)$}

In this subsection we conclude the proof of Proposition \ref{propo12} and hence of the implications (b) $\Rightarrow$ (a) and (c) $\Rightarrow$ (a)
in Theorem \ref{teo1} and of Theorem \ref{teo2} by showing the following:

\begin{lemma}
The operator $\RR_\mu$ is bounded in $L^2(\mu)$.
\end{lemma}

\begin{proof}
We will argue very similarly to Semmes in \cite{Semmes}. For completeness we will show the details.

By standard Calder\'on-Zygmund theory, it is enough to prove that for any cube $Q_0\in\DD_\mu$ and any function $f$ supported on $Q_0$,
$$\int_{Q_0} |\RR_{\mu,\eta} f|\,d\mu\lesssim \left(\int |f|^2\,d\mu\right)^{1/2}\,\mu(Q_0)^{1/2}.$$
We consider the corona type decomposition and the family $\ttt$ 
given by Proposition \ref{propo12}, and then we write
\begin{equation}\label{eq109}
\int_{Q_0} |\RR_{\mu,\eta} f|\,d\mu \leq \sum_{R\in\ttt:R\subset Q_0} \int_{Q_0}|K_R(f\mu)|\,d\mu + \sum_{R\in\ttt:R\not\subset Q_0} \int_{Q_0}|K_R(f\mu)|\,d\mu,
\end{equation}
uniformly on $\eta>0$.
It is immediate to check that any summand in the last sum vanishes unless there exists 
$S\in\tree(R)$ with
$S\supset Q_0$ and $\ell(S)\approx Q_0$. So this sum has a bounded number of nonzero terms and by Lemma \ref{lemkr} we get
$$\sum_{R\in\ttt:R\not\subset Q_0} \int_{Q_0}|K_R(f\mu)|\,d\mu
\lesssim \|f\|_{L^2(\mu)}\,\mu(Q_0)^{1/2}.$$

To deal with the first sum on the right hand side of \rf{eq109} we use again  Lemma \ref{lemkr}, with $p=3/2$, and we take into account that $K_R(f\mu)=K_R(\chi_{3R}f\mu)$:
$$\sum_{R\in\ttt:R\subset Q_0} \int_{Q_0}|K_R(f\mu)|\,d\mu \lesssim 
\sum_{R\in\ttt:R\subset Q_0} \left(\;\avint_{3R} |f|^{3/2}\,d\mu\right)^{2/3}\,\mu(R).$$
Then, by the packing condition in Lemma \ref{lemcarleson} and Carleson's embedding theorem (see Theorem
5.8 in \cite{Tolsa-llibre}, for example):
$$\sum_{R\in\ttt:R\subset Q_0} \int_{Q_0}|K_R(f\mu)|\,d\mu \lesssim 
\int_{Q_0} \left(\sup_{R\ni x} \;\avint_{3R} |f|^{3/2}\,d\mu\right)^{2/3} \,d\mu(x)$$
Since the maximal operator
$$\wt M_{\mu}f(x):=\sup_{Q\in\DD_\mu:Q\ni x} \left(\;\avint_{3Q} |f|^{3/2}\,d\mu\right)^{2/3} $$
is bounded in $L^{2}(\mu)$,
we obtain
$$\sum_{R\in\ttt:R\subset Q_0} \int_{Q_0}|K_R(f\mu)|\,d\mu \lesssim \int_{Q_0} \wt M_{\mu}f \,d\mu
\leq \left(\int |f|^2 \,d\mu\right)^{1/2}\,\mu(Q_0)^{1/2},$$
as wished.
\end{proof}

\vvv

\section{The proof of Corollary \ref{coro1}}

We need the following auxiliary result.

\begin{lemma}
Let $\Omega\subset\R^{n+1}$, $n\geq 2$, be a domain with $n$-AD-regular boundary.
Let $u$ be a  non-negative, bounded,  harmonic function in $\Omega$, vanishing at $\infty$, and let $B$ be a ball centered at
$\partial\Omega$. Suppose that $u$ vanishes continuously in $\partial\Omega\setminus \frac{11}{10}B$. Then, there is a constant
$\alpha>0$ such that
\begin{equation}\label{eqdjl125}
u(x) \lesssim \frac{r(B)^{n-1+\alpha}}{\bigl(r(B) + \dist(x,B)\bigr)^{n-1+\alpha}}\,\|u\|_{L^\infty(\Omega)}.
\end{equation}
Both $\alpha$ and the constant implicit in the above estimate depend only on $n$ the AD-regularity constant of $\partial \Omega$ .
\end{lemma}

Although this result is probably quite well known we will show the details of the proof.

\begin{proof}
Assume that $\Omega$ is not bounded (the arguments when $\Omega$ is bounded are analogous). 
Further, by translating and dilating $\Omega$ if necessary we may assume that $B=B(0,1)$. Denote by $T$ the involution
$$T(x) = \frac x{|x|^2},$$
and consider the Kelvin transform
$$\wt u (x) = \frac1{|x|^{n-1}}\,u\bigl(T(x)\bigr).$$
This function is harmonic, continuous and bounded in $B\cap T(\Omega)$. Also, it vanishes in 
$B\cap\partial( T(\Omega))$ and
it is bounded by $C\,\|u\|_{L^\infty(\Omega)}$ in $\partial ( B\cap T(\Omega))$, and thus, by the maximum principle,
$$\|\wt u\|_{L^\infty(B\cap T(\Omega))} \leq C\,\|u\|_{L^\infty(\Omega)}.$$

It is easy to check that $T$ transforms $n$-AD-regular sets into $n$-AD-regular sets. Thus, 
$T(\Omega)$ satisfies the CDC, and so $\wt u$ is H\"older continuous in $\frac12 B\cap \partial( T(\Omega))$, and for some $\alpha>0$ it satisfies
$$\wt u(x)\lesssim \dist\bigl(x,\partial (T(\Omega))\bigr)^\alpha\,\|\wt u\|_{L^\infty(B\cap T(\Omega))}
\lesssim |x|^{\alpha} \,\|u\|_{L^\infty(\Omega)} \quad \mbox{ for all $x\in \frac12B\cap T(\Omega)$,}
$$
since $0\in\partial (T(\Omega))$. This is equivalent to saying that
$$u(x)\lesssim
\frac{1}{|x|^{n-1+\alpha}}\,\|u\|_{L^\infty(\Omega)}
\quad \mbox{ for all $x\in  \Omega\setminus 2B$,}
$$
which proves the lemma.
\end{proof}
\vv

\begin{proof}[Proof of Corollary \ref{coro1}]
We will show that if $\Omega$ is a corkscrew domain with $n$-AD-regular boundary and
there exists some constant $C>0$ such that
\begin{equation}\label{eq991}
\|Su\|_{L^p(\mu)} 
\leq C\,\|N_*u\|_{L^p(\mu)}\quad 
\mbox{ for any function $u\in C_0(\overline
\Omega)$, harmonic in $\Omega$,}
\end{equation}
then the assumption (c) in Theorem \ref{teo1} holds for functions that, besides being bounded and harmonic, belong to
$C_0(\overline \Omega)$. By Remark \ref{rem***} this is enough to prove the Key Lemma \ref{lemclau} and thus
the uniform rectifiability of $\partial\Omega$, since the assumption (c) is not used
elsewhere in the proof of the implication  (c) $\Rightarrow$ (a) of
Theorem \ref{teo1}.

So we have to show that there exists some $C>0$ such that if $u\in C_0(\overline\Omega)$ is harmonic 
in $\Omega$ and $B$ is a ball centered at $\partial\Omega$, then
\begin{equation}\label{eqdk57}
\int_B |\nabla u(x)|^2\,\dist(x,\partial\Omega)\,dx\leq C\,\|u\|^2_{L^\infty(\Omega)}\,r(B)^n.
\end{equation}
To prove this, let $u\in C_0(\overline\Omega)$ be harmonic in $\Omega$ and consider a continuous nonnegative function
$\vphi_B$ which equals $1$ in $\frac52B$ and vanishes in $(3B)^c$, with $\|\vphi\|_\infty\leq1$. Then, write
$$u(x) = \int \vphi_B\, u\,d\omega^x + \int (1-\vphi_B)\, u\,d\omega^x =: u_1(x) + u_2(x).$$
Note that $u_1$ and $u_2$ are harmonic in $\Omega$, continuous in $\overline\Omega$, and vanishing at $\infty$, and
$\|u_i\|_{L^\infty(\Omega)}\leq \|u\|_{L^\infty(\Omega)}$ for $i=1,2$.

To deal with the non-local function $u_2$ we just take into account that 
$u_2$ vanishes in $\partial \Omega\cap \frac52B$ and apply Caccioppoli's inequality.
For the application of Caccioppoli's inequality, note that $u_2$ is harmonic in $\Omega$, subharmonic in $\frac52B$ (when extended by $0$ to $\frac52B\setminus \overline\Omega$) and that
$u_2\in W^{1,2}(\frac52 B)$, because it vanishes continuously in $\frac52B\cap\partial\Omega$. Then we get
$$\int_B|\nabla u_2|^2\,dx\lesssim \frac1{r(2B)^2}\int_{2B}|u_2|^2\,dx\lesssim \|u\|_{L^\infty(\Omega)}^2 \,r(B)^{n-1},$$
which implies that 
$$\int_B |\nabla u_2(x)|^2\,\dist(x,\partial\Omega)\,dx\lesssim\|u\|^2_{L^\infty(\Omega)}\,r(B)^n.$$

To prove the analogous estimate for $u_1$, first we use Fubini and H\" older's inequality, and then we apply \rf{eq991} to $u_1$:
\begin{align*}
\int_B |\nabla u_1(x)|^2\,\dist(x,\partial\Omega)\,dx & \lesssim \int_{2B} \int_{y\in\Gamma(x)}|\nabla u_1(y)|^2\,\dist(y,\partial\Omega)^{1-n}\,dy\, d\mu(x)\\
& \lesssim \| Su_1\|^{2}_{L^p(\mu)} \, \mu(B)^{1-\frac{2}{p}}\\
&\lesssim \|N_*u_1\|^{2}_{L^p(\mu)} \,  r(B)^{n-\frac{2n}{p}}.
\end{align*}
From the estimate \rf{eqdjl125} we deduce that
$$N_*u_1(x)\lesssim\frac{r(3B)^{n-1+\alpha}}{\bigl(r(3B) + \dist(x,3B)\bigr)^{n-1+\alpha}}\,\|u_1\|_{L^\infty(\Omega)},$$
which, in turn, implies that 
\begin{align}\label{eqfju44}
\|N_*u_1\|_{L^p(\mu)}^p& \lesssim \|u_1\|_{L^\infty(\Omega)}^p\int
\left(\frac{r(3B)^{n-1+\alpha}}{\bigl(r(3B) + \dist(x,3B)\bigr)^{n-1+\alpha}}\right)^p\,d\mu(x)\\
&\lesssim \|u_1\|_{L^\infty(\Omega)}^p\,\mu(B),\nonumber
\end{align}
where we took into account that $n>2$ and so $(n-1+\alpha)p>n$ to estimate the last integral.
Therefore,
$$\int_B |\nabla u_1(x)|^2\,\dist(x,\partial\Omega)\,dx \lesssim
\|u_1\|_{L^\infty(\Omega)}^2\,r(B)^n\leq \|u\|_{L^\infty(\Omega)}^2\,r(B)^n,$$
and the proof of the corollary is complete.
\end{proof}
\vv

Note that in the case $n=1$ we can ensure that the integral in \rf{eqfju44} is bounded by $c\,\mu(B)$ only if we assume $p>1/\alpha$.
So arguing as above we derive:

\begin{coro}\label{coro2}
Let $\Omega\subset\R^2$ be a corkscrew domain with $1$-AD-regular boundary. 
There exists some constant $\alpha>0$ depending only on the AD-regularity constant of $\partial\Omega$
such that the following holds.
Suppose that for some $p > 1/\alpha$ there exists some constant $C_p>0$ such that
\begin{equation*}
\|Su\|_{L^p(\mu)} 
\leq C_p\,\|N_*u\|_{L^p(\mu)}\quad 
\mbox{ for any function $u\in C_0(\overline\Omega)$ harmonic in $\Omega$.}
\end{equation*}
Then $\partial\Omega$ is $1$-uniformly rectifiable.
\end{coro}
\vvv

\vvv

\end{document}